\documentclass[11pt]{article}
\usepackage{amssymb}
\usepackage{amsmath}
\usepackage{mathrsfs}
\usepackage{graphics}
\usepackage{graphicx}
\usepackage{dsfont}
\usepackage{subfigure}
\usepackage[T1]{fontenc}
\usepackage{latexsym,amssymb,amsmath,amsfonts,amsthm}\usepackage{txfonts}
\newcommand{\R}{{\mathbb R}}

\newcommand{\N}{{\mathbb N}}

\newcommand{\ds}{\displaystyle}
\newcommand{\no}{\nonumber}
\newcommand{\be}{\begin{eqnarray}}
\newcommand{\ben}{\begin{eqnarray*}}
\newcommand{\en}{\end{eqnarray}}
\newcommand{\enn}{\end{eqnarray*}}
\newcommand{\ba}{\backslash}
\newcommand{\pa}{\partial}

\newcommand{\ov}{\overline}

\newcommand{\I}{{\rm Im}}
\newcommand{\Rt}{{\rm Re}}
\newcommand{\g}{\gamma}
\newcommand{\G}{\Gamma}

\newcommand{\vep}{\varepsilon}

\newcommand{\om}{\omega}

\newcommand{\wid}{\widetilde}

\newcommand{\ol}{\overline}

\newcommand{\half}{\frac{1}{2}}

\newtheorem{theorem}{Theorem}[section]
\newtheorem{lemma}[theorem]{Lemma}

\newtheorem{remark}[theorem]{Remark}

\newtheorem{algorithm}{Algorithm}[section]

\begin{document}
\renewcommand{\theequation}{\arabic{section}.\arabic{equation}}
\begin{titlepage}
\title{\bf A novel integral equation for scattering by locally rough surfaces
and application to the inverse problem}
\author{Haiwen Zhang,\ \ Bo Zhang\\
LSEC and Institute of Applied Mathematics\\
Academy of Mathematics and Systems Science\\
Chinese Academy of Sciences\\
Beijing 100190, China\\
({\sf zhanghaiwen@amss.ac.cn},\ {\sf b.zhang@amt.ac.cn})}
\date{}
\end{titlepage}
\maketitle

\vspace{.2in}

\begin{abstract}
This paper is concerned with the direct and inverse acoustic or electromagnetic scattering
problems by a locally perturbed, perfectly reflecting, infinite plane (which is called a
locally rough surface in this paper). We propose a novel integral equation formulation
for the direct scattering problem which is defined on a bounded curve
(consisting of a bounded part of the infinite plane containing the local perturbation and
the lower part of a circle) with two corners. This novel integral equation can be solved
efficiently by using the Nystr\"{o}m method with a graded mesh introduced previously
by Kress and is capable of dealing with large wavenumber cases.
For the inverse problem, we propose a Newton iteration method to reconstruct
the local perturbation of the plane from multiple frequency far-field data,
based on the novel integral equation formulation.
Numerical examples are carried out to demonstrate that our reconstruction method is
stable and accurate even for the case of multiple-scale profiles.

\vspace{.2in}
\end{abstract}

\section{Introduction}\label{sec1}

Consider problems of scattering of plane acoustic or electromagnetic waves by a locally
perturbed, perfectly reflecting, infinite plane (which is called a locally rough surface).
Such problems occur in many applications such as radar, remote sensing, geophysics,
medical imaging and nondestructive testing
(see, e.g. \cite{BaoGaoLi2011,BaoLin2011,BurkardPotthast2010,CZ98,DeSanto1}).

In this paper we restrict the discussion to the two-dimensional case by assuming that
the local perturbation is invariant in the $x_3$ direction.
We assume throughout that the incident wave is time-harmonic ($e^{-i\om t}$ time dependence),
so that the total wave field $u$ satisfies the Helmholtz equation
\be\label{eq1}
\Delta u+k^2u=0\quad\mbox{in}\;\;D_+.
\en
Here, $D_+:=\{(x_1,x_2)\;|\;x_2>h_\G(x_1),x_1\in\R\}$ represents a homogeneous medium
above the locally rough surface denoted by $\G:=\pa D_+=\{(x_1,x_2)\;|\;x_2=h_\G(x_1),x_1\in\R\}$
with some smooth function $h_\G\in C^2(\R)$ having a compact support in $\R$, $k=\omega/c>0$
is the wave number, $\omega$ and $c$ are the frequency and speed of the wave in $D_+$, respectively.
Throughout, we will assume that the incident field $u^i$ is the plane wave
\ben
u^i(x)=\exp({ikd\cdot x}),
\enn
where $d=(\sin\theta,-\cos\theta)\in S_-$ is the incident direction, $\theta$ is the angle of
incidence, measured from the $x_2-$axis with $-\pi/2<\theta<\pi/2,$
and $S_-:=\{x=(x_1,x_2)\;|\;|x|=1,x_2<0\}$ is the lower part of the unit circle
$S=\{x\in\R^2\;|\;|x|=1\}$.
We further assume that
the total field $u(x)=u^i(x)+u^r(x)+u^s(x)$ vanishes on the surface $\G$:
\be\label{eq2}
u(x)=u^i(x)+u^r(x)+u^s(x)=0\qquad\mbox{on}\;\;\G,
\en
where $u^r$ is the reflected wave by the infinite plane $x_2=0$:
\ben
u^r(x)=-\exp(ik[x_1\sin\theta+x_2\cos\theta])
\enn
and $u^s$ is the unknown scattered wave to be determined which is required to satisfy
the Sommerfeld radiation condition
\be\label{rc}
\lim_{r\to\infty}r^{\frac12}\left(\frac{\pa u^s}{\pa r}-iku^s\right)=0,\quad r=|x|,\quad x\in D_+.
\en
This problem models scattering of electromagnetic plane waves by a locally perturbed,
perfectly conducting, infinite plane in the TE polarization case; it also models acoustic
scattering by a one-dimensional sound-soft surface.
Figure \ref{fig4_nr} presents the problem geometry.


The well-posedness of the scattering problem (\ref{eq1})-(\ref{rc}) has been studied by using
the variational method with a Dirichlet-to-Neumann (DtN) map in \cite{BaoLin2011} or the integral
equation method in \cite{Willers1987}. In particular, it was proved in \cite{Willers1987} that
$u^s$ has the following asymptotic behavior at infinity:
\ben
u^s(x)=\frac{e^{ik|x|}}{\sqrt{|x|}}\left(u^\infty(\hat{x};d)+O\Big(\frac{1}{|x|}\Big)\right),
\qquad |x|\to\infty
\enn
uniformly for all observation directions $\hat{x}\in S_+$ with $S_+:=\{x=(x_1,x_2)\;|\;|x|=1,x_2>0\}$
the upper part of the unit circle $S$, where $u^\infty(\hat{x};d)$ is called the far field pattern of
the scattered field $u^s$, depending on the observation direction $\hat{x}$ and the incident
direction $d\in S_-$.
The integral equation formulation obtained in \cite{Willers1987} is of the second kind with a compact
integral operator defined on the local perturbation part of the infinite plane. However, it is
not suitable for numerical computation since it also involves an infinite integral over the unbounded,
unperturbed part of the infinite plane.
In \cite{BaoLin2011}, the scattering problem (\ref{eq1})-(\ref{rc}) is reformulated as an equivalent
boundary value problem in a bounded domain with a DtN map on the part in $D_+$ of a large circle
enclosing the local perturbation of the plane. This equivalent boundary value problem with a non-local
boundary condition is then solved numerically by using the integral equation approach.
However, the integral equation thus obtained involves a non-local DtN map on the semi-circle
which needs to be truncated in numerical computations.

In this paper, we propose a novel integral equation formulation for the scattering problem
(\ref{eq1})-(\ref{rc}), which is defined on a bounded curve
(consisting of a bounded part of the infinite plane containing the local perturbation and
the lower part of a circle) with two corners.
Compared with \cite{BaoLin2011} and \cite{Willers1987}, our integral equation formulation
does not involve any infinite integral or a DtN map and therefore leads to fast numerical
solution of the scattering problem including the large wavenumber cases.
In fact, our integral equation can be solved efficiently
by using the Nystr\"{o}m method with a graded mesh at the two corners introduced previously
by Kress \cite{Kress1990} (see Section \ref{sec3} below).
Furthermore, we are also interested in the {\em inverse problem} of determining the locally
rough surface from the far field pattern $u^\infty(\hat{x},d)$ for all $\hat{x}\in S_+,\;d\in S_-$.
A Newton iteration method is presented to reconstruct the locally
rough surface from multi-frequency far field data, and our novel integral equation
is applied to solve the direct scattering problem in each iteration.
From the numerical examples it is seen that multi-frequency data are necessary in order to
get a stable and accurate reconstruction of the locally rough surface.

The mathematical and computational aspects of the scattering problem (\ref{eq1})-(\ref{rc})
have been studied extensively in the case when the local perturbation is below the infinite plane
which is called the cavity problem (see, e.g. \cite{ABW,BaoGaoLi2011,BaoSun2005}
and the references quoted there) and for the case of non-local perturbations which is called
the rough surface scattering (see, e.g. \cite{CZ98,CRZ99,CM05,CHP06a,CHP06b,CE10,ZC03}).

There are many works concerning numerical solutions of the inverse problem of reconstructing
the rough surfaces from the scattered field data. For example,
a Newton method was proposed in \cite{KressTran2000} to reconstruct a local rough surface
from the far-field pattern under the condition that the local perturbation is both star-like
and {\em over} the infinite plane.
An optimization method was introduced in \cite{BaoLin} to recover a mild, local rough surface
from the scattered field measured on a straight line within one wavelength above the local
rough surface, under the assumption that the local perturbation is {\em over} the infinite plane.
In \cite{BaoLin2011}, a continuation approach over the wave frequency was developed for
reconstructing a general, local rough surface from the scattered field measured on an upper
half-circle enclosing the local perturbation, based on the choice of the descent vector field.
The reconstruction obtained in \cite{BaoLin2011} is stable and accurate due to the use of
multi-frequency near-field data (see also \cite{BaoLin2012}). It should be pointed out that
the reconstruction algorithm developed in \cite{BaoLin2011} does not work with multi-frequency
far-field data. Note that our novel integral equation formulation can also be used to
develop a similar Newton inversion algorithm with multiple frequency near-field data.
For the numerical recovery of non-local rough surfaces
we refer to \cite{BurkardPotthast2010,CL05,CGHIR,DeSanto1,DeSanto2}.
For the inverse cavity problem, the reader is referred to \cite{BaoGaoLi2011,Ma05,Li2012}.

This paper is organized as follows. In Section \ref{sec2}, a novel integral equation formulation
is proposed to solve the direct scattering problem. Section \ref{sec3} is devoted to
the numerical solution of the novel integral equation. In Section \ref{sec3+}, it is proved that
the local rough surface can be uniquely determined by the far-field pattern corresponding to
a countably infinite number of incident plane waves. The Frechet differentiability is also shown
of the far-field operator which maps the surface profile function $h_\G$ to the corresponding
far field pattern $u^\infty_k(\hat{x},d)$.
The Newton method with multi-frequency far-field data is given in Section \ref{sec4},
based on the novel integral equation solver in Section \ref{sec3}.
In Section \ref{sec5}, numerical examples are carried out to demonstrate that our reconstruction
algorithm is stable and accurate even for the case of multiple-scale profiles, which is similar
to the inversion algorithm with multi-frequency near-field data developed in \cite{BaoLin2011}.

\section{A novel integral equation formulation for the direct problem}\label{sec2}
\setcounter{equation}{0}

Let $f=-(u^i+u^r).$ Then $f$ is continuous on $\G$ and $f=0$ on $\G_0:=\{(x_1,x_2)\in\G\;|\;x_2=0\}$,
that is, $f$ has a compact support on $\G$.
The scattering problem (\ref{eq1})-(\ref{rc}) can be
reformulated as the Dirichlet problem (DP) in the following way:

Find $u^s\in C^2(D_+)\cap C(\ov{D_+})$ satisfying the Helmholtz equation (\ref{eq1}) in $D_+,$
the Sommerfeld radiation condition (\ref{rc}) and the Dirichlet boundary condition:
\be\label{dbc}
u^s=f \qquad\mbox{on}\quad\G.
\en

The following uniqueness result has been proved in \cite[Theorem 3.1]{Willers1987}
for the above Dirichlet problem (DP).

\begin{theorem}\label{thm1-uni}
The problem (DP) has at most one solution in $C^2(D_+)\cap C(\ov{D}_+)$.
\end{theorem}

The existence of solutions to the problem (DP) has also been studied in \cite{Willers1987}
using an integral equation method. However, the integral equation obtained in \cite{Willers1987}
involves an infinite integral which yields difficulties in numerical computation.
In this section, we propose a new integral equation to avoid this problem.
To this end, we introduce the following notations.
Let $B_R\coloneqq\{x=(x_1,x_2)\;|\;|x|<R\}$ be a circle with $R>0$ large enough so that
the local perturbation
$\G_p:=\G\ba\G_0=\{(x_1,h_{\G}(x_1))\;|\;x_1\in\textrm{supp}(h_{\G})\}\subset B_R$.
Then $\G_R\coloneqq\G\cap B_R$ represents the part of $\G$
containing the local perturbation $\G_p$ of the infinite plane.
Denote by $x_A:=(-R,0),x_B:=(R,0)$ the endpoints of $\G_R$.
Write $\mathbb{R}^2_\pm\coloneqq\{(x_1,x_2)\in\R^2\;|\;x_2\gtrless0\}$,
$D^\pm_R\coloneqq B_R\cap D_\pm$ and $\pa B^\pm_R\coloneqq\pa B_R\cap D_\pm$,
where $D_-:=\{(x_1,x_2)\;|\;x_2<h_\G(x_1),x_1\in\R\}.$ See Figure \ref{fig4_nr}.

\begin{figure}
\centering
\includegraphics[width=4in]{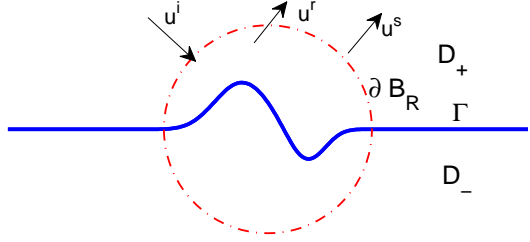}
\vspace{-0.4in}
\caption{The scattering problem from a locally rough surface}\label{fig4_nr}
\end{figure}

For $\varphi\in C(\pa D_R^-)$ define $\mathcal{S}_k$ and $\mathcal{D}_k$ to be the single-
and double-layer potentials:
\ben
(\mathcal{S}_k\varphi)(x)&:=&\int_{\pa D^-_R}\Phi_k(x,y)\varphi(y)ds(y),
 \quad x\in\mathbb{R}^2\ba\pa D^-_R\\
(\mathcal{D}_k\varphi)(x)&:=&\int_{\pa D^-_R}\frac{\pa\Phi_k(x,y)}{\pa\nu(y)}\varphi(y)ds(y),
\quad x\in\mathbb{R}^2\ba\pa D^-_R
\enn
and define $S_k,S_k^{re},K_k,K_k^{re}$ to be the boundary integral operators of the
following form:
\ben
(S_k\varphi)(x)&:=&\int_{\pa D^-_R}\Phi_k(x,y)\varphi(y)ds(y),\quad x\in\pa D^-_R\\
(S_k^{re}\varphi)(x)&:=&
     \left\{\begin{array}{l}
     \ds\int_{\pa D^-_R}\Phi_k(x,y)\varphi(y)ds(y),\quad x\in\Gamma_R\\
     \ds\int_{\pa D^-_R}\Phi_k(x^{re},y)\varphi(y)ds(y),\quad x\in\pa B^-_R\cup\{x_A,x_B\}
     \end{array}\right.\\
(K_k\varphi)(x)&:=&\int_{\pa D^-_R}\frac{\pa\Phi_k(x,y)}{\pa\nu(y)}\varphi(y)ds(y),
     \quad x\in\pa D^-_R\\
(K_k^{re}\varphi)(x)&:=&
\left\{\begin{array}{l}\ds\int_{\pa D^-_R}\frac{\pa\Phi_k(x,y)}{\pa\nu(y)}\varphi(y)ds(y),
    \quad x\in\Gamma_R\\
    \ds\int_{\pa D^-_R}\frac{\pa\Phi_k(x^{re},y)}{\pa\nu(y)}\varphi(y)ds(y),
    \quad x\in\pa B^-_R\cup\{x_A,x_B\}
    \end{array}\right.
\enn
where $x^{re}=(x_1,-x_2)$ is the reflection of $x=(x_1,x_2)$ about the $x_1$-axis,
$\Phi_k(x,y)$ is the fundamental solution of the Helmholtz equation $\Delta w+k^2w=0$
with the wavenumber $k$, and $\nu$ is the unit outward normal on $\pa D^-_R$.
Note that $\Phi_0(x,y)=-{1}/({2\pi})\ln{|x-y|}$ is the fundamental solution of the Laplace equation.

\begin{remark}\label{re2_nr} {\rm
Let $\varphi\in C(\pa D^-_R)$. From \cite[Theorem 15.8b]{Henrici1986} it follows that
the single-layer potential $\mathcal{S}_k\varphi$ is continuous throughout $\R^2$.
In addition, from \cite[Section 6.5]{Kress1999} we know that the double-layer potential
$\mathcal{D}_0\varphi$ can be continuously extended from $\R^2\ba\ov{D}^-_R$ to
$\R^2\ba D^-_R$ with the limiting value
\ben
(\mathcal{D}_0\varphi)_+(x)=\left\{\begin{array}{ll}
\ds (K_0\varphi)(x)+\frac{1}{2}\varphi(x)&\mbox{for}\;x\in\pa D^-_R\ba\{x_A,x_B\}\\
\ds (K_0\varphi)(x)+\frac{\gamma(x)}{2\pi}\varphi(x)&\mbox{for}\;x\in\{x_A,x_B\}
\end{array}\right.
\enn
where $\gamma(x)$ is the interior angle at the corner $x\in\{x_A,x_B\}$.
It remains valid for $\mathcal{D}_k\varphi$ with $k>0$ since the kernel of
$\mathcal{D}_k-\mathcal{D}_0$ is weakly singular which yields that
$(\mathcal{D}_k-\mathcal{D}_0)\varphi$ is continuous throughout $\R^2$.
Thus from the jump relations, $S_k$, $S_k^{re}$, $\widetilde{K}_k$ and
$\widetilde{K}^{re}_k$ are bounded in $C(\pa D^-_R)$, where $\wid{K}_k$
and $\wid{K}^{re}_k$ are given by
\ben
(\widetilde{K}_k\varphi)(x)&:=&
     \left\{\begin{array}{ll}
       \ds(K_k\varphi)(x)&\mbox{for}\;x\in\pa D^-_R\ba\{x_A,x_B\}\\
       \ds(K_k\varphi)(x)+\left(\frac{\gamma(x)}{2\pi}-\frac{1}{2}\right)\varphi(x)&
          \mbox{for}\;x\in\{x_A,x_B\}
     \end{array}\right.\\
(\widetilde{K}^{re}_k\varphi)(x)&:=&
     \left\{\begin{array}{ll}
        \ds (K^{re}_k\varphi)(x)&\mbox{for}\;x\in\pa D^-_R\ba\{x_A,x_B\}\\
        \ds (K^{re}_k\varphi)(x)+\left(\frac{\gamma(x)}{2\pi}-\frac{1}{2}\right)\varphi(x)&
            \mbox{for}\;x\in\{x_A,x_B\}
     \end{array}\right.
\enn
In particular, $S_k-S_0$, $S^{re}_k-S^{re}_0$, $K_k-K_0$ and $K^{re}_k-K^{re}_0$
are bounded in $C(\pa D^-_R)$.
}
\end{remark}

Let $u^s$ be the solution of the problem (DP).
Then we can extend $u^s(x)$ into $\R^2_-\ba\ov{B}_R$ by reflection,
which is denoted again by $u^s(x)$, such that $u^s(x)=-u^s(x^{re})$ in $\R^2_-\ba\ov{B}_R$.
By a regularity argument (see \cite[page 88]{ColtonKress1983} or
\cite[Theorem 3.1]{Willers1987}) and the reflection principle,
we know that $u^s\in C^2(\R^2\ba\ov{D}_R^-)\cap C(\R^2\ba D_R^-)$ and
satisfies the Helmholtz equation (\ref{eq1}) in $\R^2\ba\ov{D}_R^-$.
Following the idea in \cite{ColtonKress1998}, we seek the solution $u^s$ in the form
\be\label{eq10_nr}
u^s(x)=(\mathcal{D}_k\varphi)(x)-i\eta(\mathcal{S}_k\varphi)(x),\quad\varphi\in C(\pa D_R^-),
\quad x\in\R^2\ba\pa D_R^-
\en
where $\eta\neq0$ is a real coupling parameter.
Let $\psi^{re}$ be a continuous mapping from $\pa D_R^-$ to $\pa D_R^+$ such that
\ben
\psi^{re}(x)=\left\{
\begin{array}{ll} x, &x\in\Gamma_R\\
x^{re}, &x\in\pa B_R^-\cup\{x_A,x_B\}
\end{array}\right.
\enn
Since $u^s(x)+u^s(\psi^{re}(x))=-2\big[u^i(x)+u^r(x)\big]$ on $\Gamma_R$ and
$u^s(x)+u^s(\psi^{re}(x))=0$ on $\pa B_R^-\cup\{x_A,x_B\}$, and by the jump relations
of the layer potentials, we obtain the boundary integral equation $P\varphi(x)=g(x)$,
where
\be\label{eq11_nr}
P\varphi\coloneqq\left\{\begin{array}{ll}
\ds\varphi+\left(K_k\varphi-i\eta S_k\varphi\right)
   +\left(K^{re}_k\varphi-i\eta S_k^{re}\varphi\right), &x\in\Gamma_R\\
\ds\half\varphi+\left(K_k\varphi-i\eta S_k\varphi\right)
   +\left(K^{re}_k\varphi-i\eta S^{re}_k\varphi\right), &x\in\pa B_R^-\cup\{x_A,x_B\}
\end{array}\right.
\en
and
\be\label{eq12_nr}
g(x):=\left\{\begin{array}{ll}
-2\big(u^i(x)+u^r(x)\big), &x\in\Gamma_R\\
0, &x\in\pa B_R^-\cup\{x_A,x_B\}
\end{array}\right.
\en
Here, we have used the fact that the interior angles $\g(x)$ at the corners $x_A,x_B$
are both $\pi/2$. Note that $g\in C(\pa D_R^-)$.
Further, from Remark \ref{re2_nr} and the continuity of $\psi^{re}$ it follows that $P$ is a
bounded linear operator in $C(\pa D_R^-)$.

Conversely, we have the following result.

\begin{lemma}\label{le3_nr}
Assume that $u^s$ is of the form (\ref{eq10_nr}) with $\varphi\in C(\pa D_R^-)$
which satisfies the integral equation $P\varphi=g$ with $P$, $g$ defined in (\ref{eq11_nr}),
(\ref{eq12_nr}), respectively. Then $u^s\in C^2(D_+)\cap C(\ov{D}_+)$ and solves
the problem (DP).
\end{lemma}

\begin{proof}
Since $\varphi\in C(\pa D_R^-)$, it follows from Remark \ref{re2_nr} that
$u^s\in C^2(\R^2\ba\ov{D}_R^-)\cap C(\R^2\ba D_R^-)$ and satisfies the Helmholtz equation
(\ref{eq1}) in $\R^2\ba\ov{D}_R^-$. In addition, $P\varphi=g$ implies that
$u^s(x)+u^s(x^{re})=0$ on $\pa B_R^-\cup\{x_A,x_B\}$ and
$u^s(x)=-\big(u^i(x)+u^r(x)\big)$ on $\Gamma_R$.

Let $\tilde{u}^s(x)=-u^s(x^{re})$ in $\R^2\ba B_R$. Then $\tilde{u}^s$ satisfies
the Helmholtz equation (\ref{eq1}) in $\R^2\ba\ov{B}_R$ and
$\tilde{u}^s(x)=u^s(x)$ on $\pa B_R$. Moreover, the uniqueness of the exterior Dirichlet
problem (see, e.g. \cite[Chapter 3]{ColtonKress1998}) implies that $\tilde{u}^s=u^s$ in
$\R^2\ba B_R$. In particular, $u^s(x)=0$ on $\G\ba\G_R$ which yields that
$u^s(x)=-(u^i+u^r)$ on $\Gamma$. The proof is thus completed.
\end{proof}

We now prove the unique solvability of the integral equation $P\varphi=g$.

\begin{theorem}\label{thm1-ie}
The integral equation $P\varphi=g$ has a unique solution $\varphi\in C(\pa D_R^-)$
satisfying the estimate
\be\label{eq13_thm1}
||\varphi||_{C(\pa D_R^-)}\leq C ||u^i+u^r||_{C(\G)}
\en
\end{theorem}

\begin{proof}
We need to deal with the corners.
Following the idea in \cite{ColtonKress1998}, we introduce the following boundary integral
operators: for $z=x_A,x_B$
\ben
K_{0,z}\varphi(x)&\coloneqq&\int_{\pa D_R^-}\frac{\pa\Phi_0(x,y)}{\pa\nu(y)}\varphi(z)ds(y),
  \quad x\in\pa D_R^-\\
K^{re}_{0,z}\varphi(x)&\coloneqq&\left\{\begin{array}{l}
\ds\int_{\pa D_R^-}\frac{\pa\Phi_0(x,y)}{\pa\nu(y)}\varphi(z)ds(y),\quad x\in\Gamma_R\\
\ds\int_{\pa D_R^-}\frac{\pa\Phi_0(x^{re},y)}{\pa\nu(y)}\varphi(z)ds(y),
  \quad x\in\pa B_R^-\cup\{x_A,x_B\}
  \end{array}\right.
\enn
For $z=x_A,x_B$ define $B_\vep(z)\coloneqq\{x\in\R^2\;|\;|x-z|<\vep\}$ with radius $\vep$
small enough such that $B_\vep(z)\cap(\G\ba\G_0)=\emptyset$.
Choose a cut-off function $\chi\in C^\infty_0(\R^2)$ satisfying that $0\leq\chi\leq 1$,
$\chi(x)=\chi(x^{re})$, $\chi=1$ in $B_\vep(x_A)$ and $\chi=0$ in $B_\vep(x_B)$.
Since $K_{0,z}\varphi$ vanishes in $\R^2\ba\ov{D}_R^-$ for $z=x_A,x_B$,
we can rewrite (\ref{eq10_nr}) in the following form:
\ben
u^s(x)&=&\chi(x)\left[\int_{\pa D_R^-}\left(\frac{\pa\Phi_k(x,y)}{\pa\nu(y)}
         -i\eta\Phi_k(x,y)\right)\varphi(y)ds(y)
         -\int_{\pa D_R^-}\frac{\pa\Phi_0(x,y)}{\pa\nu(y)}\varphi(x_A)ds(y)\right]\\
      &&+[1-\chi(x)]\left[\int_{\pa D_R^-}\left(\frac{\pa\Phi_k(x,y)}{\pa\nu(y)}
        -i\eta\Phi_k(x,y)\right)\varphi(y)ds(y)\right.\\
      &&\qquad\qquad\left.-\int_{\pa D_R^-}\frac{\pa\Phi_0(x,y)}{\pa\nu(y)}\varphi(x_B)ds(y)\right],
        \quad x\in\R^2\ba\ov{D}_R^-
\enn
Accordingly, using the jump relations of the layer potentials and the fact that $\chi(x)=\chi(x^{re})$,
we rewrite $P\varphi$, defined in (\ref{eq11_nr}), as
$P\varphi=I_{\chi}\varphi+A\varphi+B\varphi$, where
\ben
(I_{\chi}\varphi)(x)&\coloneqq&\left\{\begin{array}{ll}
\ds\varphi(x), &x\in\Gamma_R\\
\ds(1/2)[\varphi(x)+\chi\varphi(x_A)+(1-\chi)\varphi(x_B)], &x\in\pa B_R^-\cup\{x_A,x_B\}
      \end{array} \right.\\
(A\varphi)(x)&\coloneqq& -\chi(x)\varphi(x_A)-[1-\chi(x)]\varphi(x_B)\\
(B\varphi)(x)&\coloneqq& \chi(x)\left(K_k\varphi-i\eta S_k\varphi-K_{0,x_A}\varphi\right)(x)
                         +[1-\chi(x)]\left(K_k\varphi-i\eta S_k\varphi-K_{0,x_B}\varphi\right)(x)\\
             &&\qquad+\chi(x)\left(K^{re}_k\varphi-i\eta S^{re}_k\varphi
               -K^{re}_{0,x_A}\varphi\right)(x)+[1-\chi(x)]\left(K^{re}_k\varphi
               -i\eta S^{re}_k\varphi-K^{re}_{0,x_B}\varphi\right)(x)
\enn
Here, $I_{\chi},A$ are bounded in $C(\pa D_R^-)$. From Remark \ref{re2_nr},
$B$ is also bounded in $C(\pa D_R^-)$.

{\bf Step 1.} We show that $P$ is a Fredholm operator of index zero.

Let
\ben
(M_0\varphi)(x)&\coloneqq&\chi(x)\left(K_0\varphi-K_{0,x_A}\varphi\right)(x)
     +[1-\chi(x)]\left(K_0\varphi-K_{0,x_B}\varphi\right)(x)\\
    &&+\chi(x)\left(K^{re}_0\varphi-K^{re}_{0,x_A}\varphi\right)(x)
      +[1-\chi(x)]\left(K^{re}_0\varphi-K^{re}_{0,x_B}\varphi\right)(x)
\enn
From Remark \ref{re2_nr} it is easy to see that $M_0$ is bounded in $C(\pa D_R^-)$.
Since the integral operator $B-M_0$ has a weakly singular kernel and $\psi^{re}$ is a
continuous mapping, the operator $B-M_0$ is compact in $C(\pa D_R^-)$.
Thus, $P-(I_\chi+M_0)=A+B-M_0$ is compact in $C(\pa D_R^-)$.

Moreover, for $z=x_A,x_B$ and $0<r<\vep$ choose a cut-off function $\psi_{r,z}\in C^\infty_0(\R^2)$
satisfying that $0\leq\psi_{r,z}\leq 1$, $\psi_{r,z}(x)=1$ in the region $0\leq|x-z|\leq r/2$ and
$\psi_{r,z}(x)=0$ in the region $r\leq|x-z|<\infty$.
Define $M_{0,r}:\;C(\pa D_R^-)\rightarrow C(\pa D_R^-)$ by
\ben
M_{0,r}\varphi&\coloneqq&\psi_{r,x_A}\chi\Big[K_0\big(\psi_{r,x_A}\varphi\big)
     -K_{0,x_A}\big(\psi_{r,x_A}\varphi(x_A)\big)\Big]\\
   &&+\psi_{r,x_B}(1-\chi)\Big[K_0\big(\psi_{r,x_B}\varphi\big)
     -K_{0,x_B}\big(\psi_{r,x_B}\varphi(x_B)\big)\Big]\\
   &&+\psi_{r,x_A}\chi\Big[K^{re}_0\big(\psi_{r,x_A}\varphi\big)
     -K^{re}_{0,x_A}\big(\psi_{r,x_A}\varphi(x_A)\big)\Big]\\
   &&+\psi_{r,x_B}(1-\chi)\Big[K^{re}_0\big(\psi_{r,x_B}\varphi\big)
     -K^{re}_{0,x_B}\big(\psi_{r,x_B}\varphi(x_B)\big)\Big],\qquad \varphi\in C(\pa D_R^-).
\enn
Since the kernel of $M_{0,r}-M_0$ vanishes in a neighborhood of $(x_A,x_A)$ and $(x_B,x_B)$,
it is compact in $C(\pa D_R^-)$. Thus $P-(I_{\chi}+M_{0,r})$ is compact in $C(\pa D_R^-)$
since $P-(I_{\chi}+M_0)$ is compact in $C(\pa D_R^-)$.

We now introduce the following norm on $C(\pa D_R^-):$
\ben
||\varphi||_{\infty,0}&:=&\max\left\{\max\limits_{\G_R}\Big[\left|\chi\left(\varphi
    -\varphi(x_A)\right)\right|+\left|(1-\chi)\left(\varphi-\varphi(x_B)\right)\right|
    +|\varphi(x_A)|+|\varphi(x_B)|\Big]\right.,\\
&&\left.\max\limits_{\pa B_R^-\cup\{x_A,x_B\}}\left[\left|\half\chi\left(\varphi
   -\varphi(x_A)\right)\right|+\left|\half(1-\chi)\left(\varphi-\varphi(x_B)\right)\right|
   +|\varphi(x_A)|+|\varphi(x_B)|\right]\right\}
\enn
which is equivalent to the maximum norm $||\cdot||_\infty$. It is easy to see that
$I_{\chi}$ is a bijection from $C(\pa D_R^-)$ to $C(\pa D_R^-)$ with
\ben
I_{\chi}^{-1}\psi=\left\{
\begin{array}{ll}
\ds \psi, &x\in\G_R\\
\ds 2\psi-\chi\psi(x_A)-(1-\chi)\psi(x_B), &x\in\pa B_R^-\cup\{x_A,x_B\}
\end{array}\right.
\enn
and $||I_{\chi}\varphi||_\infty\leq||\varphi||_{\infty,0}$ for all $\varphi\in C(\pa D_R^-)$.
Furthermore, choose a function $\phi_0\in C(\pa D_R^-)$ such that $\phi_0\geq0$,
$\phi_0(x_A)=\phi_0(x_B)=0$ and $\phi_0$ reaches its maximum on $\G_R$.
Then $||I_{\chi}\phi_0||_{\infty}=||\phi_0||_{\infty,0}$, which implies that
$||I_{\chi}||_{C(\pa D_R^-,||\cdot||_{\infty,0})\rightarrow C(\pa D_R^-,||\cdot||_\infty)}=1$.

We now prove that $||M_{0,r}||_{C(\pa D_R^-,||\cdot||_{\infty,0})\rightarrow
C(\pa D_R^-,||\cdot||_\infty)}<1$ for $r>0$ small enough.
For any $\varphi\in C(\pa D_R^-)$, $\textrm{supp}(M_{0,r}\varphi)\subset\ov{B_r(x_A)\cup B_r(x_B)}$,
so we only need to consider $(M_{0,r}\varphi)(x)$
for $x\in\pa D_R^-\cap\big(B_r(x_A)\cup B_r(x_B)\big)$.
Note first that for $x\in\pa D_R^-\cap B_r(x_A)$ we have
\ben
(M_{0,r}\varphi)(x)&=&\psi_{r,x_A}\Big[K_0\big(\psi_{r,x_A}\varphi\big)
-K_{0,x_A}\big(\psi_{r,x_A}\varphi(x_A)\big)\Big]\\
&=&\psi_{r,x_A}\Big[K^{re}_0\big(\psi_{r,x_A}\varphi\big)
-K^{re}_{0,x_A}\big(\psi_{r,x_A}\varphi(x_A)\big)\Big].
\enn
Then the required estimate can be obtained by following the idea in \cite[Section 3.5]{ColtonKress1998},
together with the inequality:
\be\label{eq17_nr}
|\nu(y)\cdot(x-y)|\leq C|x-y|^2
\en
for $x\in\G_R\cup\{x_A,x_B\},\;y\in\G_R$ or for $x\in\pa B_R,\;y\in\pa B_R^-.$

When $x=x_A$, using (\ref{eq17_nr}) we have
\ben
|(M_{0,r}\varphi)(x_A)|\leq C r||\varphi||_{\infty,0}
\enn

When $x\in\G_R\cap B_r(x_A)$, the integral $(M_{0,r}\varphi)(x)$ can be split up into two parts:
the first one over $\G_R\cap B_r(x_A)$ and the second one over $\pa B_R^-\cap B_r(x_A)$.
The first part can be bounded with the upper bound $Cr||\varphi||_{\infty,0}$, estimated directly
using (\ref{eq17_nr}).
Noting that $\nu(y)\cdot(x-y)$ does not change its sign for $x\in\G_R\cap B_r(x_A)$ and
$y\in\pa B_R^-\cap B_r(x_A)$, we have that for $x\in\G_R\cap B_r(x_A)$ the second part is bounded by
\ben
&&2\left|\psi_{r,x_A}(x)\int_{\pa B_R^-}\left[\frac{\pa\Phi_0(x,y)}{\pa\nu(y)}
      \psi_{r,x_A}(y)\big(\varphi(y)-\varphi(x_A)\big)\right]ds(y)\right|\\
&&\qquad\le2\int_{\pa B_R^-\cap B_r(x_A)}\left|\frac{\pa\Phi_0(x,y)}{\pa\nu(y)}\right|ds(y)
      ||\varphi||_{\infty,0}\\
&&\qquad=2\left|\int_{\pa B_R^-\cap B_r(x_A)}\frac{\pa\Phi_0(x,y)}{\pa\nu(y)}ds(y)\right|
   ||\varphi||_{\infty,0}\\
&&\qquad=\frac{\alpha(x)}{\pi}||\varphi||_{\infty,0}
\enn
where $\alpha(x)$ is the angle between the two segments connecting $x$ and the two endpoints of
the arc $\pa B_R^-\cap B_r(x_A)$.
Since the interior angle at $x_A$ is $\pi/2$, we have $\alpha(x)\le3\pi/4$ if $r$ is small enough.
Thus, for $x\in\G_R\cap B_r(x_A)$ we have
\ben
|(M_{0,r}\varphi)(x)|\leq (C r+3/4)||\varphi||_{\infty,0}
\enn

When $x\in\pa B_R^-\cap B_r(x_A)$, the integral $(M_{0,r}\varphi)(x)$ can also be split up into
two parts: the first one over $\G_R\cap B_r(x_A)$ and the second one over $\pa B_R^-\cap B_r(x_A)$.
Since ${\pa\ln|x-y|}/{\pa\nu(y)}={-x_2}/{|x-y|^2}$ for $y\in\G_R\cap B_r(x_A)$,
then, for $x\in\pa B_R^-\cap B_r(x_A)$ the first part is equal to
\ben
&&\psi_{r,x_A}(x)\int_{\G_R}\left[\frac{\pa\Phi_0(x,y)}{\pa\nu(y)}\psi_{r,x_A}(y)
   \big(\varphi(y)-\varphi(x_A)\big)\right]ds(y)\\
&&\quad+\psi_{r,x_A}(x)\int_{\G_R}\left[\frac{\pa\Phi_0(x^{re},y)}{\pa\nu(y)}\psi_{r,x_A}(y)
   \big(\varphi(y)-\varphi(x_A)\big)\right]ds(y)=0,
\enn
whilst the second part can be estimated using (\ref{eq17_nr}) and is
bounded with the upper bound $Cr||\varphi||_{\infty,0}$.
Thus, for $x\in\pa B_{R-}\cap B_r(x_A)$ we have
\ben
|(M_{0,r}\varphi)(x)|\leq C r||\varphi||_{\infty,0}
\enn

Similar estimates can be obtained for $(M_{0,r}\varphi)(x)$ when $x\in\pa D_R^-\cap B_r(x_B)$.
Thus we have the estimate
\ben
|(M_{0,r}\varphi)(x)|\leq (C r+3/4)||\varphi||_{\infty,0}
\enn
for $x\in\pa D_R^-\cap\big(B_r(x_A)\cup B_r(x_B)\big)$.
Choosing $r>0$ small enough we obtain that
$||M_{0,r}||_{C(\pa D_R^-,||\cdot||_{\infty,0})\rightarrow C(\pa D_R^-,||\cdot||_{\infty})}<1$.
Then, by the Neumann series, $I_\chi+M_{0,r}$ has a bounded inverse in
$C(\pa D_R^-)$, yielding that $P$ is a Fredholm operator of index zero
since $P=[P-(I_\chi+M_{0,r})]+I_\chi+M_{0,r}$ and $P-(I_\chi+M_{0,r})$ is compact in $C(\pa D_R^-)$.

{\bf Step 2.} We prove that $P$ is injective and therefore invertible in $C(\pa D_R^-)$.

Let $P\varphi=0$ for $\varphi\in C(\pa D_R^-)$. Then, by Lemma \ref{le3_nr}, the scattered field $u^s$
defined by (\ref{eq10_nr}) vanishes on the boundary $\G$. From Theorem \ref{thm1-uni}
and Holmgren's uniqueness theorem, $u^s$ vanishes in $\R^2\ba{D_R^-}$.
On the other hand, by the jump relations of the layer potentials (see Remark \ref{re2_nr}),
we have $(P_0\varphi)(x)=0$ for $x\in\pa D_R^-$,
where $(P_0\varphi)(x):=\alpha(x)\varphi(x)+(K_k\varphi)(x)-i\eta(S_k\varphi)(x)$
is bounded in $C(\pa D_R^-)$ with $\alpha(x)={1}/{2}$ for $x\in\pa D_R^-\ba\{x_A,x_B\}$
and $\alpha(x)={\g(x)}/({2\pi})$ for $x=x_A,\;x_B$.
It was proved in \cite{Ruland1978} that $P_0$ has a bounded inverse in $C(\pa D_R^-)$,
yielding that $\varphi=0$.

By the Fredholm alternative, $P$ is invertible with a bounded inverse $P^{-1}$ in $C(\pa D_R^-)$.
The proof is thus completed.
\end{proof}

Combining Lemma \ref{le3_nr}, Theorems \ref{thm1-uni} and \ref{thm1-ie} and the mapping
properties of the single- and double-layer potentials in the space of continuous functions,
we get the following result on the well-posedness of the problem (DP).

\begin{theorem}\label{thm1-wp}
The problem (DP) has a unique solution $u^s\in C^2(D_+)\cap C(\ov{D}_+)$. Furthermore,
\be\label{eq13_nr}
||u^s||_{C(\ol{D}_+)}\leq C ||u^i+u^r||_{C(\Gamma)}.
\en
\end{theorem}

\begin{remark}\label{re3} {\rm
By the asymptotic behavior of the fundamental solution $\Phi_k$, we have the following far
field pattern for the scattered field $u^s$ given by (\ref{eq10_nr}):
\be\label{eq14_nr}
u^\infty(\hat{x})=\frac{e^{-i\pi/4}}{\sqrt{8\pi k}}\int_{\pa D_R^-}
[k\nu(y)\cdot\hat{x}+\eta]e^{-ik\hat{x}\cdot y}\varphi(y)ds(y)
\en
which is an analytic function on the unit circle $S$, where $\varphi$ is the solution of the
integral equation $P\varphi=g$.
}
\end{remark}

\section{Numerical solution of the novel integral equation}\label{sec3}
\setcounter{equation}{0}

We make use of the Nystr\"{o}m method with a graded mesh introduced in
\cite[Section 3.5]{ColtonKress1998} (see also \cite{Kress1990}) to solve the integral
equation $P\varphi=g$. Let $\pa D_R^-$ be parameterized as
$x(s)=(x_1(s),x_2(s)),\;0\leq s\leq 2\pi$ such that $(x_1(0),x_2(0))=x_B$ and
$(x_1(\pi),x_2(\pi))=x_A$, where
\ben
x_1(s)&=&\left\{\begin{array}{ll}
\ds -{2R}\omega(s)/{\pi}+R, &0\le s\le\pi\\
\ds R\cos(\omega(s)), &\pi<s\leq2\pi
    \end{array}\right.\\
x_2(s)&=&\left\{\begin{array}{ll}
\ds h_\G(-{2R}\omega(s)/{\pi}+R), &0\leq s\leq\pi\\
\ds R\sin(\omega(s)), &\pi<s\leq2\pi
\end{array}\right.
\enn
Here, $\omega:[0,2\pi]\rightarrow[0,2\pi]$ is a strictly, monotonically increasing function
satisfying that $\omega(s)=\pi{[v(s)]^p}/({[v(s)]^p+[v(\pi-s)]^p})$ for $0\leq s\leq\pi$ and
$\omega(s)=\omega(s-\pi)+\pi$ for $\pi<s\leq2\pi$, where
\ben
v(s)=\left(\frac{1}{p}-\half\right)\left(\frac{\pi-2s}{\pi}\right)^3
+\frac{1}{p}\cdot\frac{2s-\pi}{\pi}+\half
\enn
with $p=4$. Note that $\omega(0)=0,\;\omega(\pi)=\pi$ and $\omega'(0)=\omega'(\pi)=0$.
Then the integral equation $P\varphi=g$
can be rewritten as $P\varphi(x(t))=g(x(t))$, where
\ben
P\varphi(x(t))=\left\{\begin{array}{ll}
\ds \varphi(x(t))+2\int^{2\pi}_0\left[\frac{\pa\Phi_k(x(t),x(s))}{\pa\nu(x(s))}
     -i\eta\Phi_k(x(t),x(s))\right]|x'(s)|ds, &0< t<\pi\\
\ds \half\varphi(x(t))+\int^{2\pi}_0\left[\frac{\pa\Phi_k(x(t),x(s))}{\pa\nu(x(s))}
    -i\eta\Phi_k(x(t),x(s))\right]|x'(s)|ds\\
\ds \quad\quad\quad+\int^{2\pi}_0\left[\frac{\pa\Phi_k(x^{re}(t),x(s))}{\pa\nu(x(s))}
    -i\eta\Phi_k(x^{re}(t),x(s))\right]|x'(s)|ds, &\pi\leq t\leq 2\pi
    \end{array}\right.
\enn
For $j=0,1,\ldots,2n-1$ with $n\in\mathds{N}$ and $n>0$, let $t_j=j\pi/n$ and $g_j=g(x(t_j))$.
Then we get the approximation value $\varphi^{(n)}_j$ of $\varphi$ at the points $x(t_j)$
by solving the linear system
\ben
\half\varphi^{(n)}_i+2\sum_{j=1,j\neq n}^{2n-1}\left(R^{(n)}_{|i-j|}K_1(t_i,t_j)+\frac{\pi}{n}
    K_2(t_i,t_j)\right)\varphi^{(n)}_j&=&g_i,\qquad i=0,n\\
\varphi^{(n)}_i+2\sum_{j=1,j\neq n}^{2n-1}\left(R^{(n)}_{|i-j|}K_1(t_i,t_j)+\frac{\pi}{n}
    K_2(t_i,t_j)\right)\varphi^{(n)}_j&=&g_i,\qquad i=1,\cdots,n-1\\
\varphi^{(n)}_i+\sum_{j=1,j\neq n}^{2n-1}\left(R^{(n)}_{|i-j|}K_1(t_i,t_j)+\frac{\pi}{n}
    K_2(t_i,t_j)+\frac{\pi}{n} K_3(t_i,t_j)\right)\varphi^{(n)}_j&=&g_i,\qquad i=n+1,\cdots,2n-1
\enn
where, for $j=0,1,\dots,2n-1$
\ben
R^{(n)}_j\coloneqq-\frac{2\pi}{n}\sum^{n-1}_{m=1}\frac{1}{m}\cos\left(\frac{mj\pi}{n}\right)
-\frac{(-1)^j\pi}{n^2}
\enn
and for $i=0,1,\dots,2n-1,\;j=1,2,\cdots,n-1,n+1,\cdots,2n-1$
\ben
K(t_i,t_j)&\coloneqq&\left[\frac{\pa\Phi_k(x(t_i),x(t_j))}{\pa\nu(x(t_j))}
   -i\eta\Phi_k(x(t_i),x(t_j))\right]|x'(t_j)|\\
K_1(t_i,t_j)&\coloneqq&-\frac{1}{4\pi}\left[\frac{\pa J_0(k|x(t_i)-x(t_j)|)}{\pa\nu(x(t_j))}
   -i\eta J_0(k|x(t_i)-x(t_j)|)\right]|x'(t_j)|\\
K_2(t_i,t_j)&\coloneqq& K(t_i,t_j)-K_1(t_i,t_j)\ln\left(4\sin^2\frac{t_i-t_j}{2}\right)\\
K_3(t_i,t_j)&\coloneqq&\left[\frac{\pa\Phi_k(x^{re}(t_i),x(t_j))}{\pa\nu(x(t_j))}
   -i\eta\Phi_k(x^{re}(t_i),x(t_j))\right]|x'(t_j)|
\enn
Here, $J_0$ is the Bessel function of order $0$. Further, (\ref{eq14_nr}) can be rewritten as
\ben
u^\infty(\hat{x})=\frac{e^{-i\pi/4}}{\sqrt{8\pi k}}\int_{0}^{2\pi}
  \left[k\nu(x(s))\cdot\hat{x}+\eta\right]e^{-ik\hat{x}\cdot x(s)}|x'(s)|\varphi(x(s))ds
\enn
Then for $n_f\in\mathds{N}$ with $n_f>0$ we get the approximation values of the far field pattern
at the points $\hat{x}_i={i\pi}/{n_f},$ $i=0,1,\cdots,n_f,$ by the following quadrature rule:
\ben
u^\infty(\hat{x}_i)\backsimeq\frac{e^{-i\pi/4}}{\sqrt{8\pi k}}\cdot\frac{\pi}{n}
   \sum^{2n-1}_{j=1,j\neq n}\left[k\nu(x(t_j))\cdot\hat{x}_i
   +\eta\right]e^{-ik\hat{x}_i\cdot x(t_j)}|x'(t_j)|\varphi^{(n)}_j
\enn

\section{The inverse problem}\label{sec3+}
\setcounter{equation}{0}

The {\em inverse problem} we are interested in is, given the far field pattern $u^\infty$
of the scattered wave $u^s$ of the scattering problem (\ref{eq1})-(\ref{rc})
(or the problem (DP)) for one or a finite number of incident plane waves $u^i$,
to determine the unknown locally rough surface $\G$ (or the local perturbation $\G_p$).


We have the following uniqueness theorem which can be proved by arguing similarly as in
the proof of Theorem 3.1 in \cite{Ma05}.

\begin{theorem}\label{thm-uni}
Assume that $\G_1$ and $\G_2$ are two locally rough surfaces and
$u^\infty_1(\hat{x},d)$ and $u^\infty_2(\hat{x},d)$ are the far field patterns
corresponding to $\G_1$ and $\G_2$, respectively.
If $u^\infty_1(\hat{x},d_n)=u^\infty_2(\hat{x},d_n)$ for all $\hat{x}\in S_+$
and $d_n\in S_-$ with $n\in\N$ and a fixed wave number $k$,
then $\Gamma_1=\Gamma_2$.
\end{theorem}

Given the incident plane wave $u^i(x)=\exp(ikd\cdot x)$ with the incident direction
$d\in S_-,$ we define the far field operator $F_d$ mapping the function $h_\G$ which
describes the locally rough surface $\G$ to the corresponding far field pattern
$u^\infty_k(\hat{x},d)$ in $L^2(S_+)$ of the scattered wave $u^s$ of the scattering
problem (\ref{eq1})-(\ref{rc}):
\be\label{FFO}
F_d(h_\G)=u^\infty_k(\cdot,d).
\en
Here, we use the subscript $k$ to indicate the dependence on the wave number $k$.
In terms of this far field operator, given the far-field pattern $u^\infty_k(\hat{x},d)$,
our inverse problem consists in solving the equation (\ref{FFO}) for the unknown
function $h_\G.$ This is a nonlinear and very ill-posed operator equation.
To solving this equation by the Newton method,
we need the Frechet differentiability at $h_\G$.
To this end, let $\triangle h\in C^2_{0,R}(\R):=\{h\in C^2(\R)\;|\;\textrm{supp}(h)
\subset (-R,R)\}$ be a small perturbation of the function $h_G\in C^2(\R)$
and let $\G_{\triangle h}:=\{(x_1,h_\G(x_1)+\triangle h(x_1))\;|\;x_1\in\R\}$
denote the corresponding boundary defined by $h_\G(x_1)+\triangle h(x_1)$.
Then $F_d$ is called Frechet differentiable at $h_\G$ if there exists a linear bounded
operator $F'_d(h_\G,\cdot):C^2_{0,R}(\R)\rightarrow L^2(S_+)$ such that
\ben
||F_d(h_\Gamma+\triangle h)-F_d(h_\Gamma)-F'_d(h_\G,\triangle h)||_{L^2(S_+)}
=o(||\triangle h||_{C^2(\R)})
\enn
as $||\triangle h||_{C^2(\R)}\to 0.$

\begin{theorem}\label{thm-FD}
Let $u(x,d)=u^i(x,d)+u^r(x,d)+u^s(x,d)$, where $u^s$ solves the problem (DP)
with the boundary data $f=-(u^i+u^r)$. If $h_{\G}\in C^2$, then $F_d$ is Frechet
differentiable at $h_\G$ and the derivative $F'_d(h_\G,\triangle h)=u'_\infty$
for $\triangle h\in C^2_{0,R}(\R)$. Here, $u'_\infty$ is the far field pattern of
$u'$ which solves the problem (DP) with the boundary data
$f=-(\nu_2\triangle h){\pa u}/{\pa\nu}$, where $\nu_2$ is the second component of
the unit normal $\nu$ on $\G$ directed into the infinite domain $D_+$.
\end{theorem}

\begin{proof}
The proof is similar to that of Theorem 4.1 in \cite{BaoLin2011} with appropriate
modifications.
\end{proof}

\section{The Newton method with multi-frequency data}\label{sec4}
\setcounter{equation}{0}

We now describe the Newton iteration method for solving our inverse problem of
reconstructing the function $h_\G$ from the far field data, that is, for solving
the equation (\ref{FFO}). Motivated by \cite{BaoLin2011}, we use multi-frequency
far field data in order to get an accurate reconstruction of the function $h_\G$.

For each single frequency data with wave number $k>0$, we replace (\ref{FFO}) by
the linearized equation
\be\label{LFFO}
F_d(h_\G)+F_d'(h_\G,\triangle h)=u^\infty_k(\cdot,d)
\en
which we will solve for $\triangle h$ by using the Levenberg-Marquardt algorithm
(see, e.g. \cite{Hohage1999}) in order to improve an approximation to the function $h_\G$.
The Newton method consists in iterating this procedure.

In the numerical examples, we consider the noisy measurement data
$u^\infty_{\delta,k}(\hat{x},d_l)$, $\hat{x}\in S_+,l=1,\ldots,n_d,$
which satisfies
\ben
||u^\infty_{\delta,k}(\cdot,d_l)-u^\infty_k(\cdot,d_l)||_{L^2(S_+)}
\leq\delta||u^\infty_k(\cdot,d_l)||_{L^2(S_+)}.
\enn
Here, $\delta$ is called a noisy ratio.
In practical computations $h_\G$ has to be taken from a finite-dimensional subspace
$R_M\subset C^2_{0,R}(\R)$ and the equation (\ref{LFFO}) has to be approximately solved
by projecting it on a finite-dimensional subspace of $L^2(S_+)$ by collocation
at a finite number $n_f$ of equidistant points $\hat{x}_j\in S_+$, $j=1,\ldots,n_f.$
Let $R_M=\textrm{span}\{\phi_1,\phi_2,\cdots,\phi_M\}$, where $\phi_i,j=1,2,\ldots,M,$
are spline functions with support in $(-R,R)$ (see Remark \ref{re4} below).
Assume that $h^{app}\in R_M$ is an approximation to $h_\G$ with $\G^{app}$ being the
corresponding boundary. Then, by the strategy in \cite{Hohage1999}, we seek an updated
function $\triangle h=\sum^M_{i=1}\triangle a_i\phi_i$ in $R_M$ of $h^{app}$
such that $\triangle a_i$, $i=1,\ldots,M,$ solve the minimization problem:
\be\label{eq15_nr}
\min_{\triangle a_i}\left\{\sum^{n_d}_{l=1}\sum_j^{n_f}\big|F_{d_l}(h^{app})(\hat{x}_j)
+F'_{d_l}(h^{app},\triangle h)(\hat{x}_j)
-u^\infty_{\delta,k}(\hat{x}_j,d_l)\big|^2+\beta\sum_{i=1}^M|\triangle a_i|^2\right\}
\en
where the regularization parameter $\beta>0$ is chosen such that
\be\no
&&\left(\sum^{n_d}_{l=1}\sum_j^{n_f}\big|F_{d_l}(h^{app})(\hat{x}_j)
+F'_{d_l}(h^{app},\triangle h)(\hat{x}_j)
-u^\infty_{\delta,k}(\hat{x}_j,d_l)\big|^2\right)^{1/2}\\ \label{eq16_nr}
&&\qquad\qquad=\rho\left(\sum^{n_d}_{l=1}\sum_j^{n_f}\big|F_{d_l}(h^{app})(\hat{x}_j)
-u^\infty_{\delta,k}(\hat{x}_j,d_l)\big|^2\right)^{1/2}
\en
for a given constant $\rho\in(0,1)$. Then a new approximation to $h_\G$ is given
as $h^{app}+\triangle h$. Define the error function
\ben
Err_k=\frac{1}{n_d}\sum^{n_d}_{l=1}
\left[\sum_j^{n_f}\big|F_{d_l}(h^{app})(\hat{x}_j)
     -u^\infty_{\delta,k}(\hat{x}_j,d_l)\big|^2\Big/
\sum_j^{n_f}\big|u^\infty_{\delta,k}(\hat{x}_j,d_l)\big|^2\right]^{1/2}
\enn
Then the iteration is stopped if $Err_k\leq\tau\delta$,
where $\tau>1$ is a fixed constant. See \cite{Hohage1999} for details.

\begin{remark}\label{re4} {\rm
For a positive integer $M\in\mathds{N}^+$ let $h=2R/(M+5)$ and $t_i=(i+2)h-R$.
Then the spline basis functions of $R_M$ are defined by
$\phi_i(t)=\phi((t-t_i)/h),i=1,2,\ldots,M,$ where
\ben
\phi(t):=\sum^{k+1}_{j=0}\frac{(-1)^j}{k!}\left(\begin{array}{c}k+1\\
                                      j\end{array}\right)
       \left(t+\frac{k+1}{2}-j\right)^k_+
\enn
with $z^k_+=z^k$ for $z\geq0$ and $=0$ for $z<0$.
In this paper, we choose $k=4$, that is, $\phi$ is the cubic spline function.
Note that $\phi_i\in C^3(\R)$ with support in $(-R,R)$.
See \cite{DeBoor} for details.
}
\end{remark}

\begin{remark}\label{re4a} {\rm
Our inversion Algorithm \ref{alg} below does not require the locally rough surface $\G$
to be parameterized by a function $h_\G$ since, in practical computations, $h_\G$ is taken
from the finite-dimensional subspace $R_M$ spanned by spline functions $\phi_i,j=1,2,\ldots,M$
with support in $(-R,R)$ (see discussions before Remark \ref{re4}).
Therefore, Algorithm \ref{alg} can deal with more general $C^2-$smooth $\G$.
}
\end{remark}

\begin{remark}\label{re5} {\rm
For the synthetic far-field data of the scattering problem, we choose the coupling parameter
$\eta=k$ and get a finite number of measurements $u^\infty_{\delta,k}(\hat{x}_j,d),$
$j=0,1,\ldots,n_f,$ with equidistant points $\hat{x}_j=j\pi/n_f$ for a positive integer
$n_f\in\N^+$. For the numerical solution of the scattering problem in each iteration,
we choose $\eta=0$ both to avoid the inverse crime and to reduce the complexity of the
computation. Here, we need to assume that $k$ is not a Dirichlet eigenvalue of the region
bounded by the curves $\{(x_1,h^{app}(x_1))\;|\;x_1\in[-R,R]\}$ and $\pa B_R^-$.
Further, it is seen from Theorem \ref{thm-FD} that, in order to compute the
Frechet derivative in each iteration we need to compute the normal derivative
$\pa u^s/\pa\nu$ of the scattered wave $u^s$ on a subset of
$\{x(t_0),x(t_1),\ldots,x(t_{2n-1})\}$ which is contained in
$\{(x_1,h^{app}(x_1))\;|\;x_1\in (-R,R)\}$.
Since $h^{app}\in C^3(\R)$ and the two corners $x_A$ and $x_B$ are not included in the
subset, we just use the quadrature rules in \cite{Kress1995} in the form (\ref{eq10_nr})
of the scattered field $u^s$ and the graded mesh with the discrete values
$\varphi^{(n)}_j,j=0,1,\ldots,2n-1,$ of $\varphi$ which are obtained from
the scattering problem in each iteration.
}
\end{remark}

The Newton iteration algorithm with multi-frequency far-field data can be given
in the following Algorithm \ref{alg}.

\begin{algorithm}\label{alg}
Given the far field patterns $u^\infty_{\delta,k_i}(\hat{x}_j,d_l),i=1,2,\ldots
N,j=0,1,\ldots,n_f,l=1,\cdots,n_d$, where $k_1<k_2<\cdots<k_N$.
\begin{description}
\item 1) Let $h^{app}=0$ be the initial guess of $h_\G$ and set $i=0$.
\item 2) Set $i=i+1$. If $i>N$, then stop the iteration; otherwise, set $k=k_i$
and go to Step 3).
\item 3) If $Err_k<\tau\delta$, return to Step 2); otherwise, go to Step 4).
\item 4) Solve (\ref{eq15_nr}) with the strategy (\ref{eq16_nr}) to get
an updated function $\triangle h$. Let $h^{app}$ be updated by $h^{app}+\triangle h$
and go to Step 3).
\end{description}
\end{algorithm}

\begin{remark}\label{re6} {\rm
Our novel integral equation formulation proposed in Section \ref{sec2} can also be applied
to develop a similar Newton inversion algorithm with multiple frequency near-field data.
}
\end{remark}

\section{Numerical examples}\label{sec5}
\setcounter{equation}{0}

In this section, several numerical experiments are presented to demonstrate the
effectiveness of our algorithm. The following assumptions are made in all numerical
experiments.
\begin{description}
\item 1) For each example we use multi-frequency data with the wave
numbers $k=1,3,\ldots,2N-1,$ where $N$ is the total number of frequencies.
\item 2) To generate the synthetic data and to compute the Frechet derivative
in each iteration, we solve the novel integral equation by choosing $n=128$ for the wave
number $k<13$ and $n=256$ for the wave number $k\geq13$.
\item 3) We measure the half-aperture (the measurement angle is between $0$ and $\pi$)
far-field pattern with $65$ measurement points, that is,
$n_f=64$. The noisy data $u^\infty_{\delta,k}$ are obtained as
$u^\infty_{\delta,k}=u^\infty_k+\delta\zeta||u^\infty_k||_{L^2(S_+)}/||\zeta||_{L^2(S_+)}$,
where $\zeta$ is a random number with $\Rt(\zeta),\I(\zeta)\in N(0,1)$.
\item 4) We set the parameters $\rho=0.8$ and $\tau=1.5$.
\item 5) In each figure, we use solid line '-' and dashed line '- -'
to represent the actual curve and the reconstructed curve,
respectively.
\item 6) For the shape of the local perturbation of the infinite plane in all examples,
we assume that $\textrm{supp}(h_\G)\in(-1,1)$; we further choose $R=1$ and use the
smooth curves which are not in $R_M$.
\end{description}

\textbf{Example 1.} In this example, we consider the case when the local
perturbation of the infinite plane is over the $x_1$-axis with
\ben
h_\G(x_1)=\phi(({x_1+0.2})/{0.3}),
\enn
where $\phi$ is defined in Remark \ref{re4}. Here, we consider noisy data with $3\%$
noise and use one incident direction $d=(\sin(\pi/3),-\cos(\pi/3))$.
For the inverse problem, we choose the number
of the spline basis functions to be $M=10$and the total number of frequencies to be
$N=6$. Figure \ref{fig1} shows the reconstructed curves at $k=1,5,7,11,$
respectively.
\begin{figure}[htbp]
  \centering
  \subfigure{
    \includegraphics[width=3in]{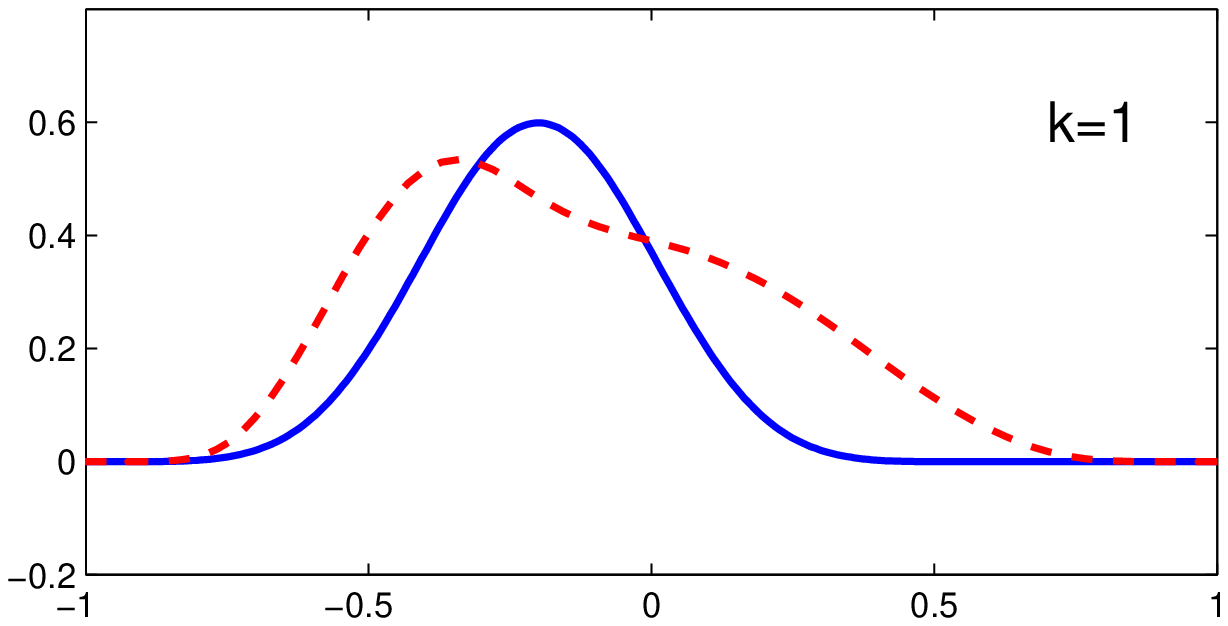}}
  \hspace{-0.35in}
  \vspace{-0.5in}
  \subfigure{
    \includegraphics[width=3in]{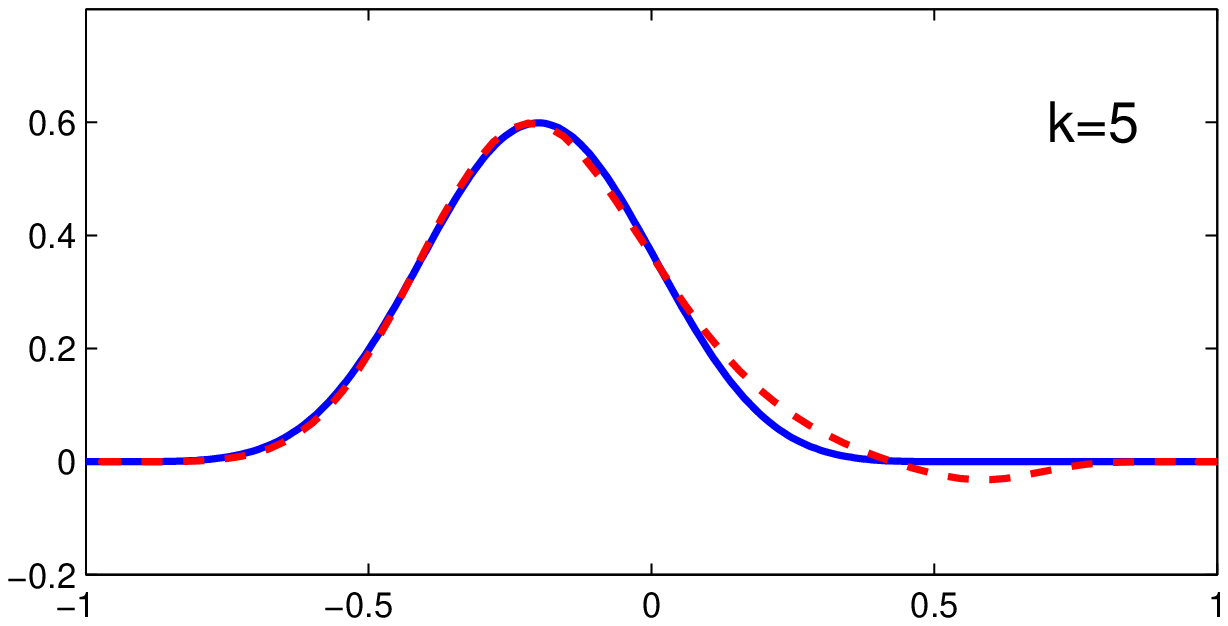}}
  \vspace{0in}
  \hspace{-0.35in}
  \subfigure{
    \includegraphics[width=3in]{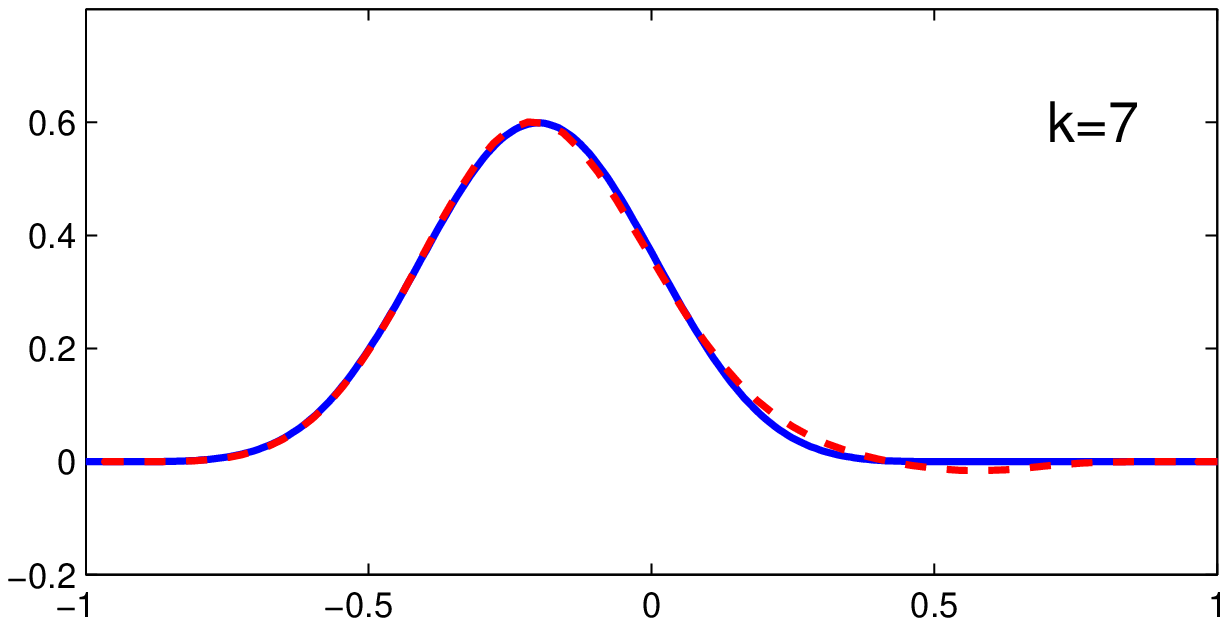}}
  \vspace{0in}
  \hspace{-0.35in}
  \subfigure{
    \includegraphics[width=3in]{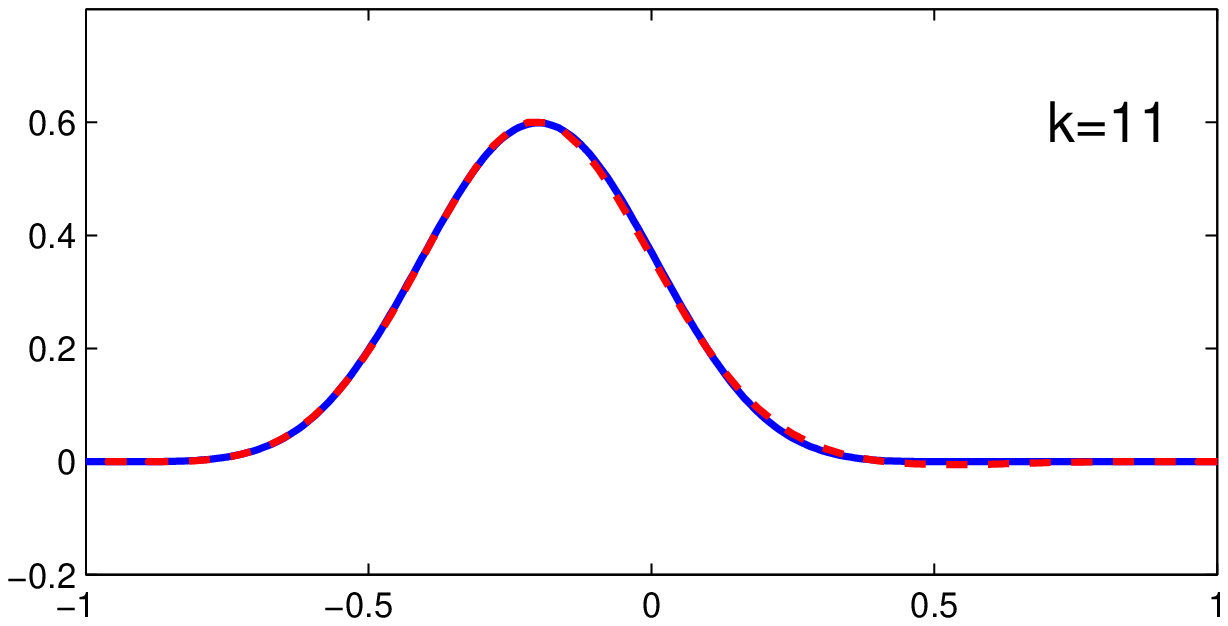}}
      \vspace{-0.5in}
\caption{The reconstructed curve (dashed line) at $k=1,5,7,11,$ respectively, from $3\%$ noisy data
with one incident direction $d=(\sin(\pi/3),-\cos(\pi/3))$, where the real curve is denoted by
the solid line.
}\label{fig1}
\end{figure}

\textbf{Example 2.}  In this example, we consider the case when the local
perturbation of the infinite plane is under the $x_1$-axis with
\ben
h_\G(x_1)=-0.8\phi(({x_1-0.3})/{0.2}),
\enn
where $\phi$ is also given in Remark \ref{re4}.
In the inverse problem, the number of the spline
basis functions is chosen to be $M=10$ and the total number
of frequencies is chosen to be $N=9$. Figure \ref{fig2} presents the
reconstructed curves at $k=1,5,11,17,$ respectively, from $3\%$ noisy data
with one incident direction $d=(\sin(\pi/3),-\cos(\pi/3))$.
\begin{figure}[htbp]
  \centering
  \subfigure{
    \includegraphics[width=3in]{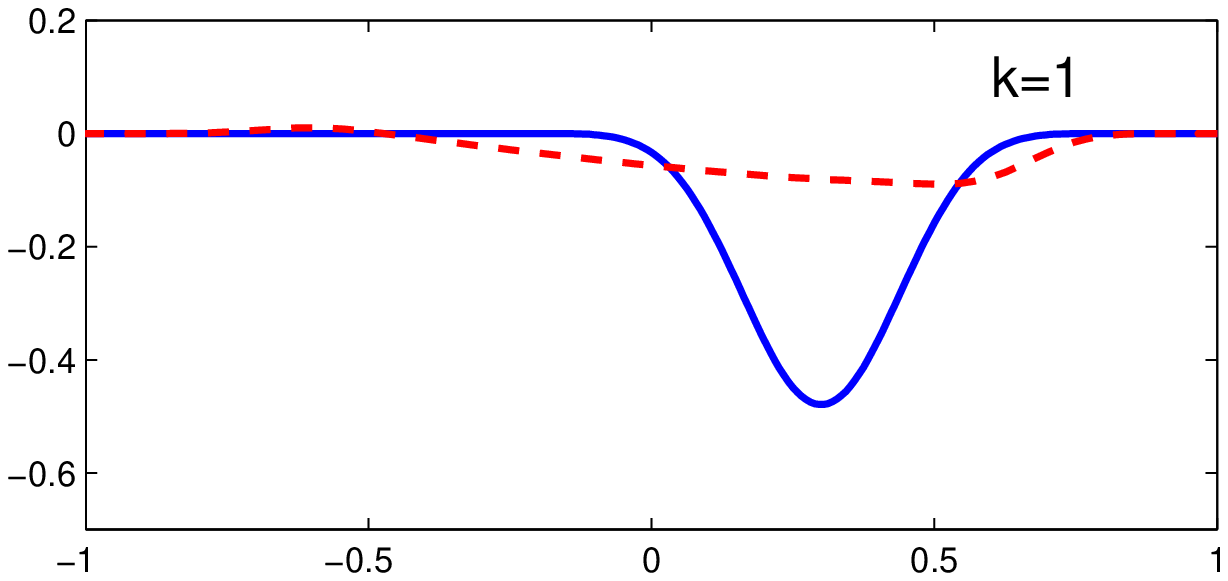}}
  \hspace{-0.35in}
  \vspace{-0.6in}
  \subfigure{
    \includegraphics[width=3in]{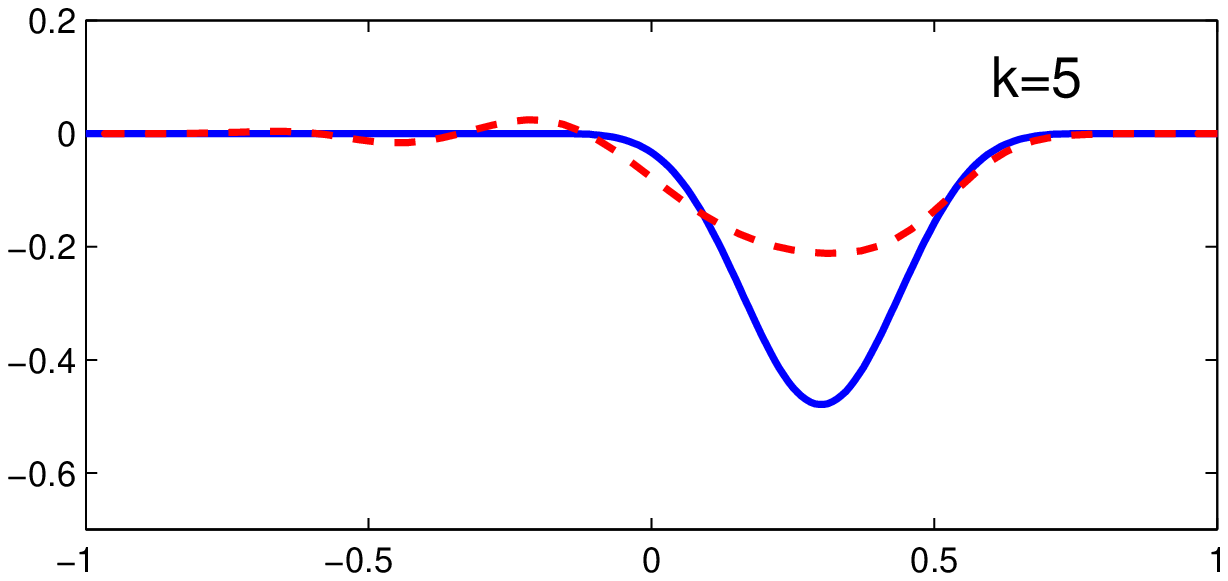}}
  \vspace{0in}
  \hspace{-0.35in}
  \subfigure{
    \includegraphics[width=3in]{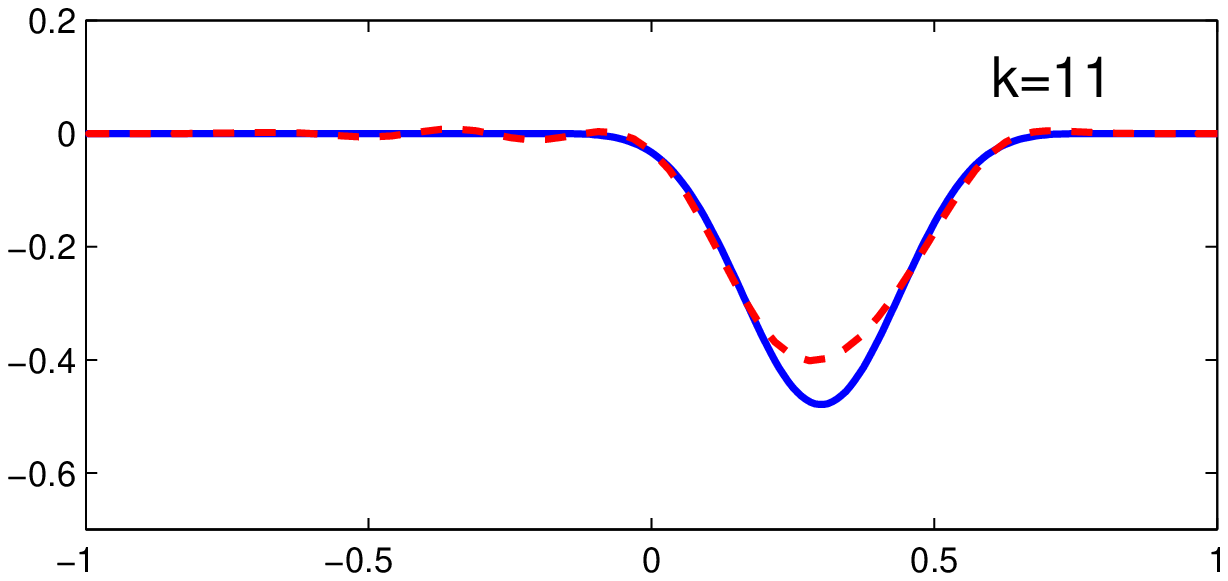}}
  \vspace{0in}
  \hspace{-0.35in}
  \subfigure{
    \includegraphics[width=3in]{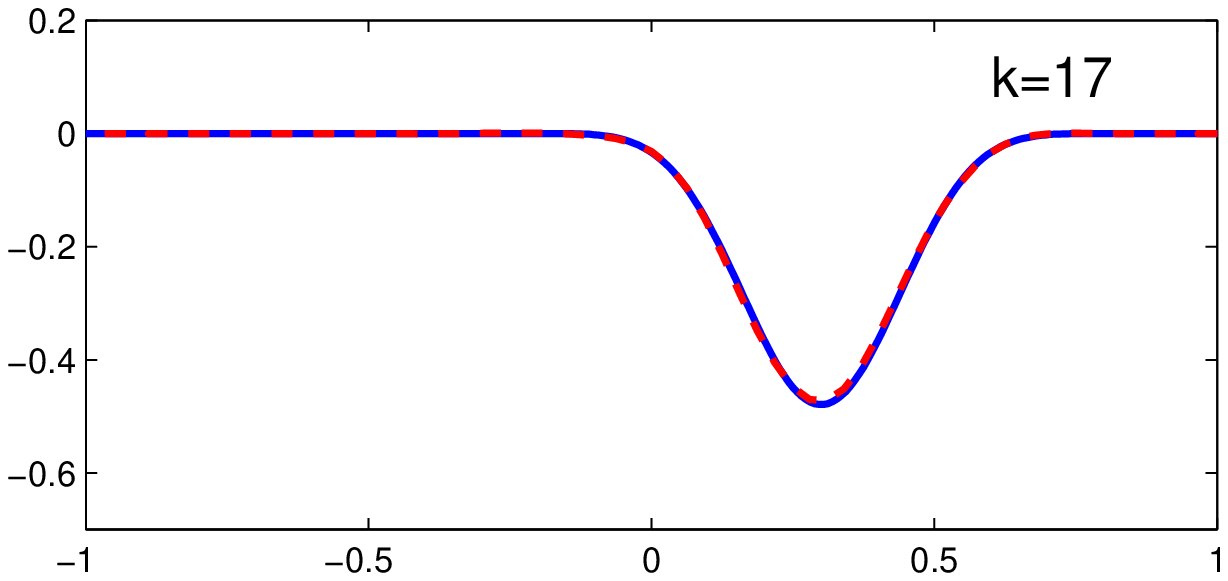}}
      \vspace{-0.5in}
\caption{The reconstructed curve (dashed line) at $k=1,5,11,17,$ respectively, from $3\%$ noisy data
with one incident direction $d=(\sin(\pi/3),-\cos(\pi/3))$, where the real curve is denoted by
the solid line.
}\label{fig2}
\end{figure}

\textbf{Example 3.} The reconstruction considered in this example is a more challenging
one with
\ben
h_\G(x_1)=\left\{\begin{array}{ll}
\ds \exp\left[16/(25x_1^2-16)\right]\sin(4\pi x_1), &|x_1|<4/5\\
\ds 0, &|x_1|\geq0
\end{array}\right.
\enn
Here, we consider $10\%$ noisy data. In order to get a good reconstruction,
the number of the spline basis functions is taken to be $M=20$ and the
total number of frequencies is taken to be $N=15$.
Figure \ref{fig3} gives the reconstruction at $k=1,9,19,29$, respectively,
with one incident direction $d=(0,-1)$ (normal incidence from the top).
\begin{figure}[htbp]
  \centering
  \subfigure{
    \includegraphics[width=3in]{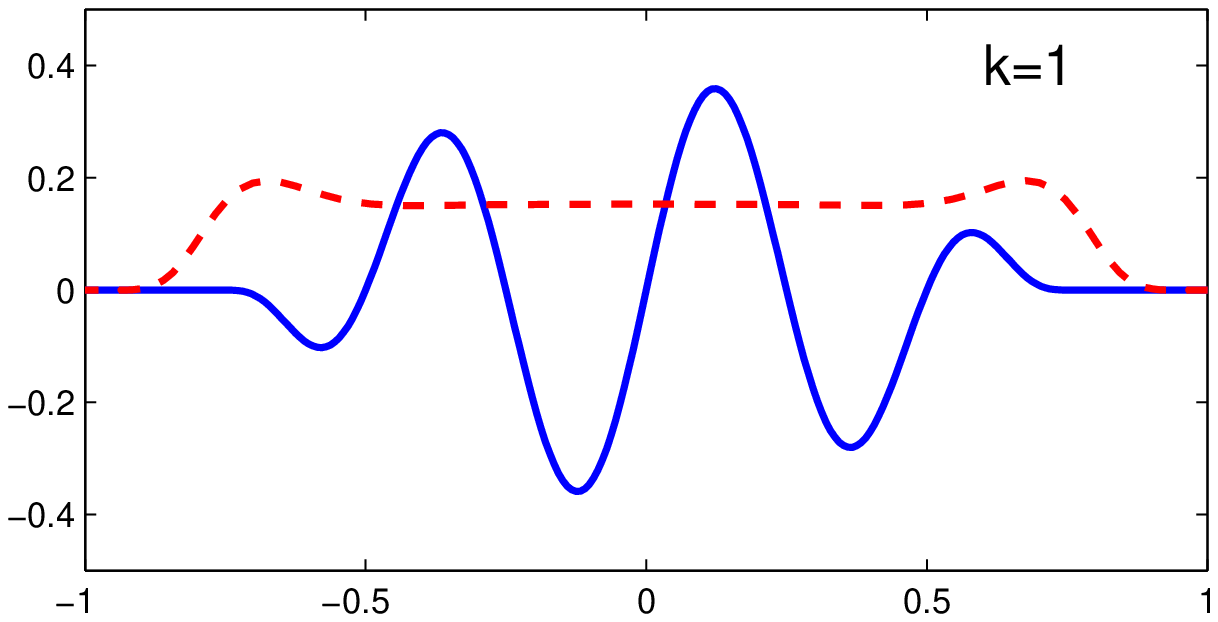}}
  \hspace{-0.35in}
  \vspace{-0.5in}
  \subfigure{
    \includegraphics[width=3in]{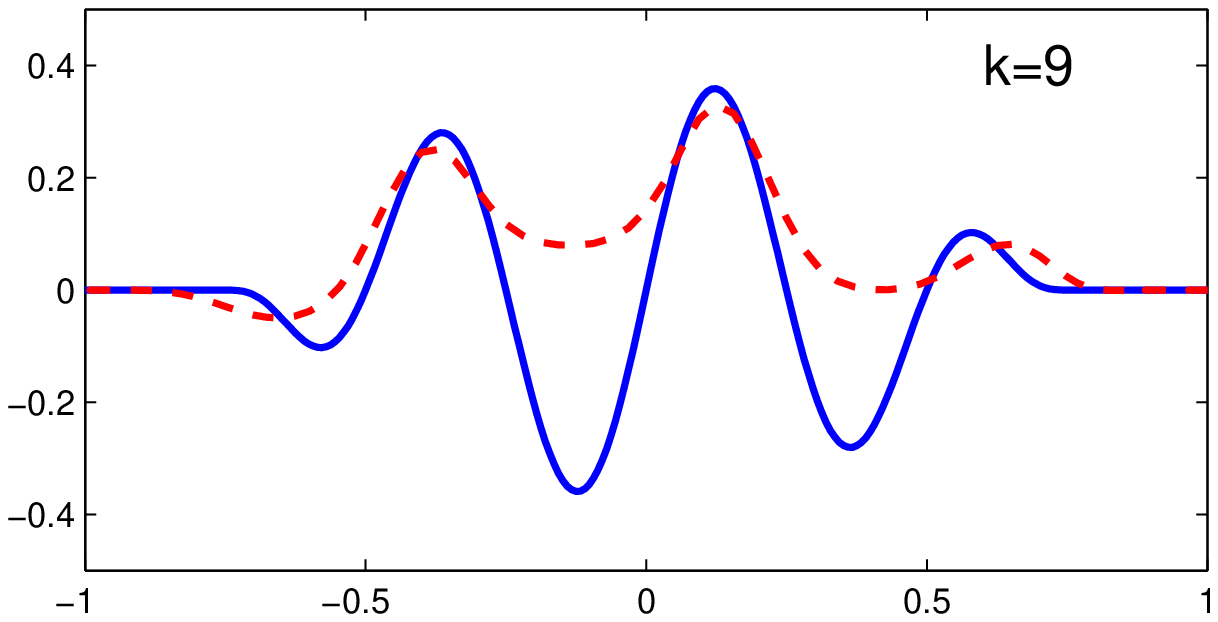}}
  \vspace{0in}
  \hspace{-0.35in}
  \subfigure{
    \includegraphics[width=3in]{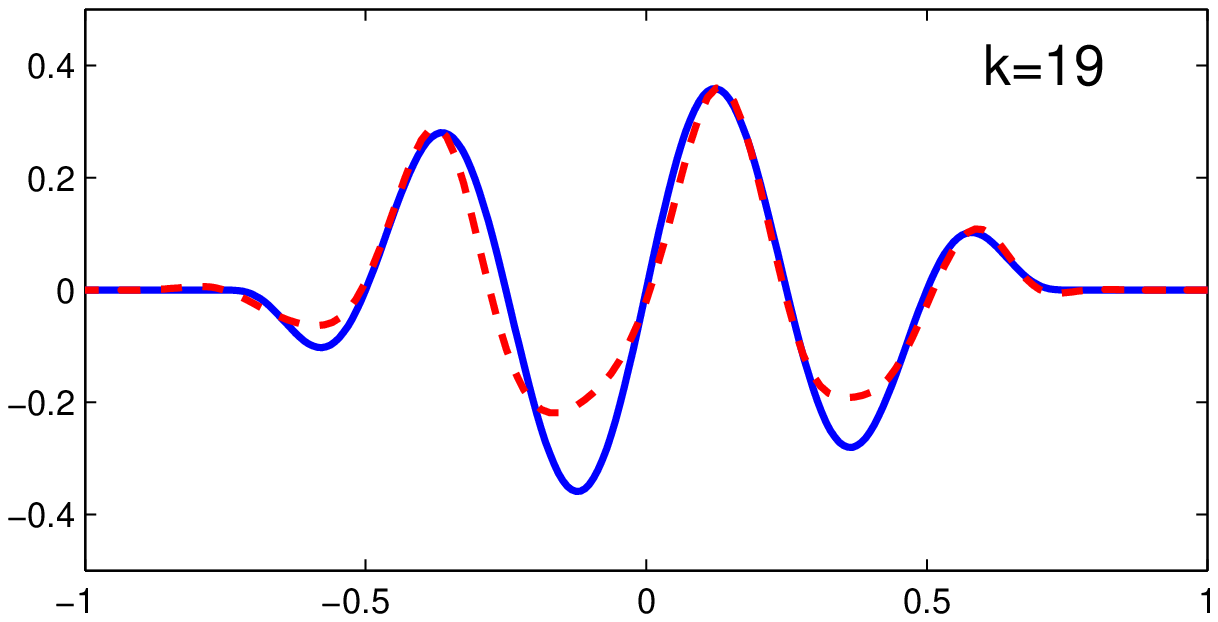}}
  \vspace{0in}
  \hspace{-0.35in}
  \subfigure{
    \includegraphics[width=3in]{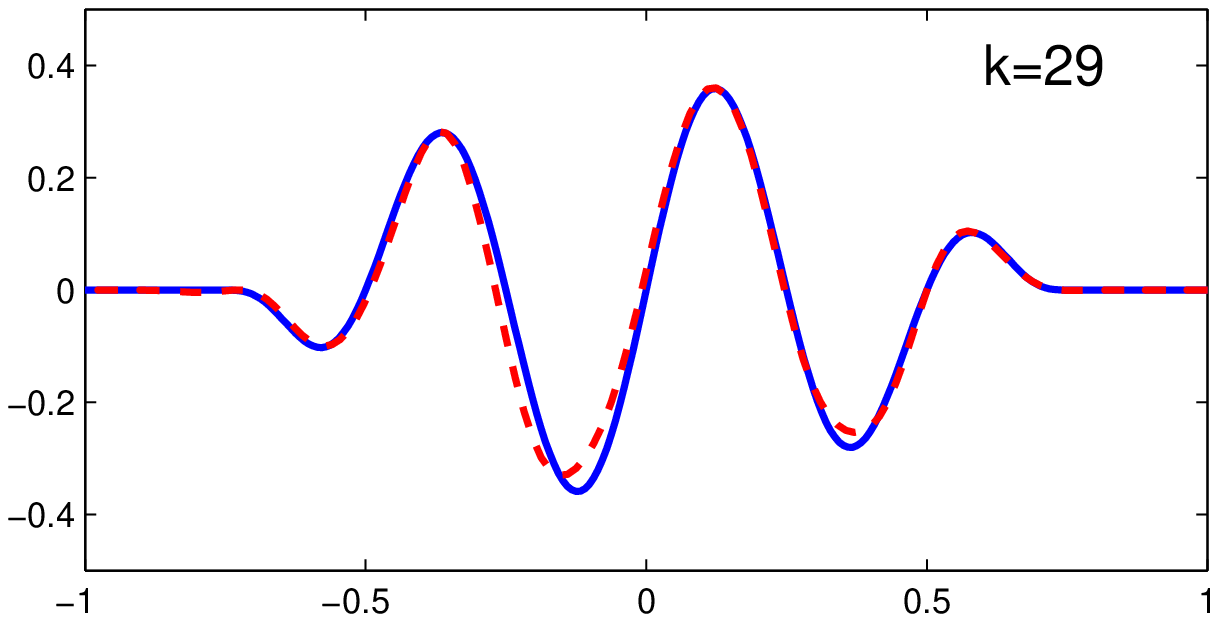}}
      \vspace{-0.5in}
\caption{The reconstructed curve (dashed line) at $k=1,9,19,29$, respectively, from $10\%$ noisy
data with normal incidence from the top, where the real curve is denoted by the solid line.
}\label{fig3}
\end{figure}


\textbf{Example 4 (multi-scale profile).}
We now consider the multi-scale case with
\ben
h_\G(x_1)=\left\{\begin{array}{ll}
\ds \exp\left[16/(25x_1^2-16)\right]\left[0.5+0.1\sin(16\pi x_1)\right], &|x_1|<4/5\\
\ds 0, &|x_1|\geq0.
\end{array}\right.
\enn
This function has two scales: the macro-scale is represented by
the function $0.5\exp[16/(25x_1^2-16)]$, and the micro-scale is represented by the
function $0.1\exp[16/(25x_1^2-16)]\sin(16\pi x_1)$.
To capture the two-scale features of the profile, the number of spline basis functions
is chosen to be $M=40$, and the total number of frequencies used is $N=30$.
The reconstruction is obtained with $10\%$ noisy data using one incident plane wave
with normal incidence from the top. Figure \ref{fig5} presents the reconstructed profiles
at $k=9,33,45,59$. From Figure \ref{fig5} it is observed that the macro-scale features are captured
when $k=9$ (Figure \ref{fig5}, top left), while the micro-scale features are captured at $k=59$
(Figure \ref{fig5}, bottom right). It is interesting to note that the resolution of the reconstruction
does not improve much for $k\in[9,33]$ and then improves from a larger $k$ (e.g., $k=45$)
until a sufficiently large $k$ (e.g., $k=59$) for which the whole local rough surface is accurately
recovered even with $10\%$ noisy data. This indicates that our Newton algorithm with multiple
frequency far-field data can give a stable and accurate reconstruction of multi-scale profiles
with noise data as long as sufficiently high frequency data are used. This is similar to
the reconstruction algorithm with multi-frequency near-field data developed in \cite{BaoLin2011}.

\begin{figure}[htbp]
  \centering
  \subfigure{
    \includegraphics[width=3in]{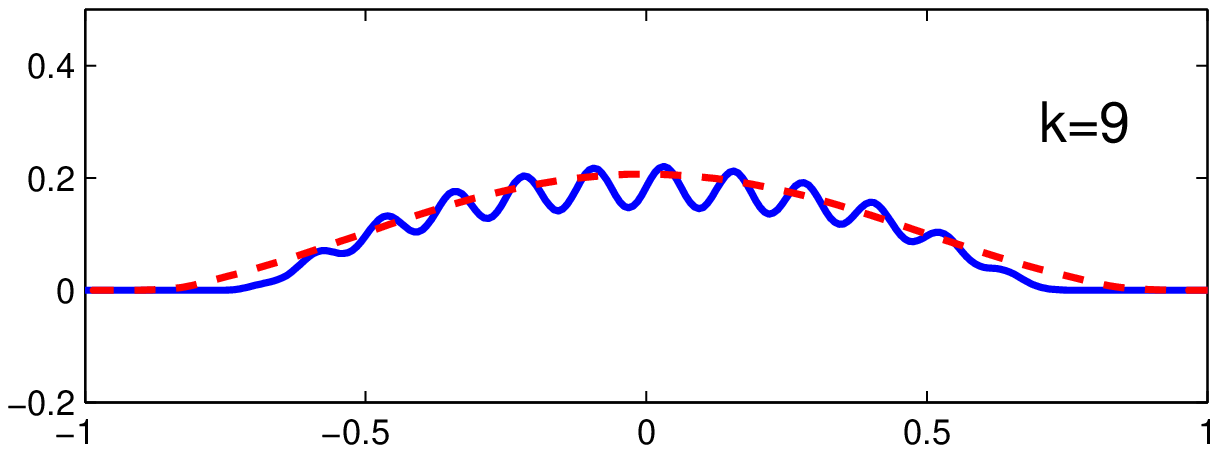}}
  \hspace{-0.35in}
  \vspace{-0.7in}
  \subfigure{
    \includegraphics[width=3in]{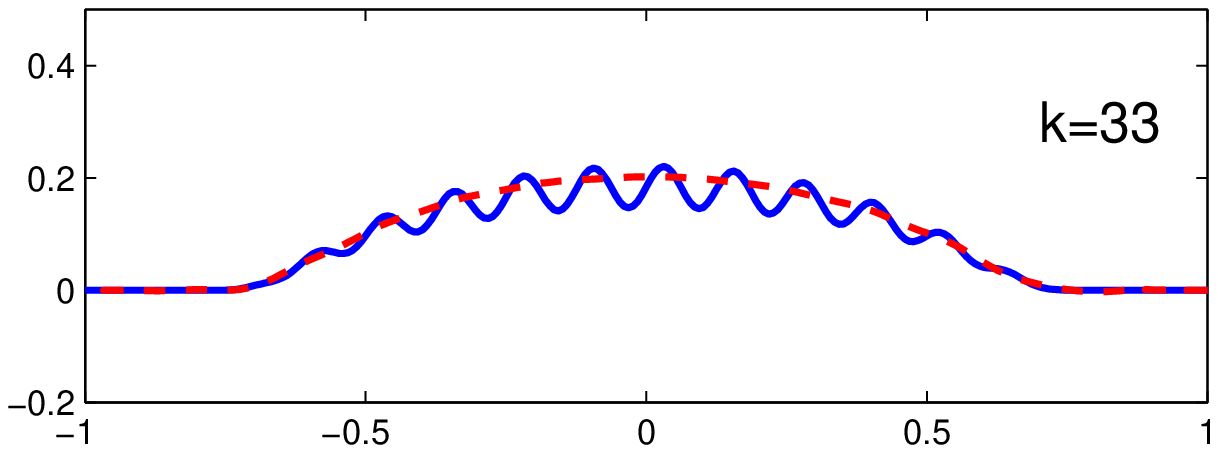}}
  \vspace{0in}
  \hspace{-0.35in}
  \subfigure{
    \includegraphics[width=3in]{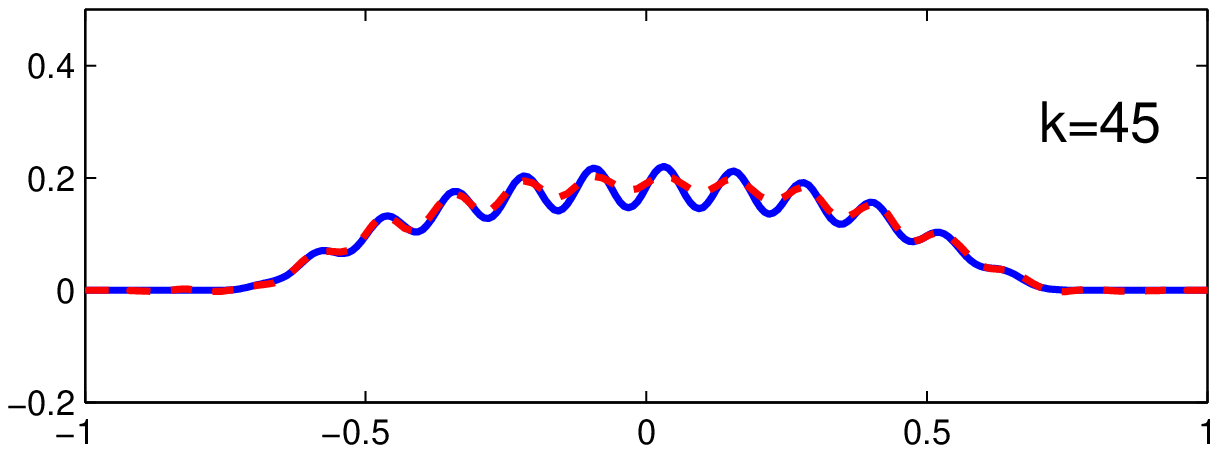}}
  \vspace{0in}
  \hspace{-0.35in}
  \subfigure{
    \includegraphics[width=3in]{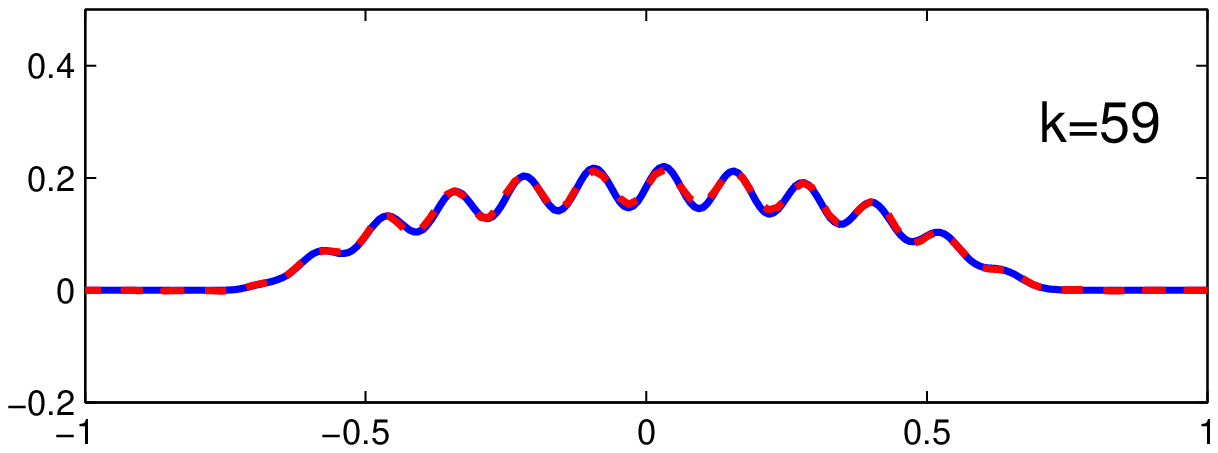}}
      \vspace{-0.5in}
  \caption{The reconstructed curve (dashed line) at $k=9,33,45,59$, respectively, from $10\%$ noisy
data with normal incidence from the top, where the real curve is denoted by the solid line.
}\label{fig5} 
\end{figure}

The above numerical results illustrate that the Newton iteration algorithm with
multiple frequency data gives a stable and accurate reconstruction of the local perturbation
of the infinite plane even in the presence of $10\%$ noise in measurements.
From Figures \ref{fig1}-\ref{fig5} it is seen that the upper part of
the locally rough surface can be recovered easily at lower frequencies; however,
much higher frequencies are needed in order to recover the deep, lower part of the
locally rough surface as well as the fine details of the micro-scale features of
multi-scale profiles.

We are currently trying to extend the technique to the TM polarization case.
Furthermore, it is anticipated that the reconstruction method can be generalized
to the three-dimensional case.

\section*{Acknowledgements}

This work was supported by the NNSF of China under grants 11071244 and 11161130002.

\end{document}